\documentclass{amsart}
\usepackage{hyperref}
\usepackage{geometry}
\usepackage{enumerate}
\usepackage{amsmath,amssymb,amsthm,mathrsfs,amsfonts,dsfont}
\usepackage{graphicx}
\usepackage{subfigure}
\usepackage{tikz}
\usetikzlibrary{trees,decorations.pathmorphing}
\usetikzlibrary{calc}
\usepackage{tikz}
\usetikzlibrary{shapes}
\usetikzlibrary{plotmarks}
\usetikzlibrary{arrows}
\usetikzlibrary{positioning}
\usepackage{tkz-graph}
\theoremstyle{definition}
\newtheorem{definition}{Definition}[section]
\newtheorem{prop}[definition]{Proposition}
\newtheorem{thm}[definition]{Theorem}
\newtheorem{cor}[definition]{Corollary}
\newtheorem{lem}[definition]{Lemma}
\newtheorem{example}[definition]{Example}
\newtheorem{remark}[definition]{Remark}
\newtheorem{problem}{Problem}
\newcommand\intertrans{\mathrel{\rightleftarrows}}
\newcommand\Lg{\mathbf{L}}
\newcommand\LL{\mathcal{L}}
\newcommand\NN{\mathbb{N}}
\newcommand\cl[1]{\mathrm{#1}}
\newcommand\Th{\mathsf{Th}}
\newcommand\Cn{\mathsf{Cn}}
\newcommand\Cz{\mathsf{Cz}}
\newcommand\tr{\mathsf{tr}}

\newcommand\Fm{\mathsf{Fm}}

\newcommand\rank{\mathsf{rank}}

\def\Cd{\cl{Cd}}  
\def\Ad{\cl{Ad}}

\def\dd{\cl{d}}
\def\ddd{\vec{\cl{d}}}

\def\dCd{\vec{\cl{Cd}}{}}

\def\X{\cl{X}}
\def\E{\cl{E}}
\def\R{\cl{R}}
\def\Ss{\cl{S}}

\newcommand{\ax}[1]{\ensuremath{\mathsf{#1}}}

\newcommand\defeq{\mathrel{\stackrel{\makebox[0pt]{\mbox{\normalfont\tiny def}}}{=}}}

\begin{document}
\title[Distances between formal theories]
      {Distances between formal theories}
 \author[Friend et al.]{Mich{\`e}le Friend, Mohamed Khaled, Koen Lefever and Gergely Sz{\'e}kely}
\begin{abstract}In the literature, there have been several methods and definitions for working out if two theories are ``equivalent'' (essentially the same) or not. In this article, we do something subtler. We provide means to measure distances (and explore connections) between formal theories. We introduce two main notions for such distances. The first one is that of \textit{axiomatic distance}, but we argue that it might be of limited interest. The more interesting and widely applicable notion is that of \textit{conceptual distance} which measures the minimum number of concepts that distinguish two theories. For instance, we use conceptual distance to show that relativistic and classical kinematics are distinguished by one concept only. We also develop further notions of distance, and we include a number of suggestions for applying and extending our project.
\end{abstract}
\maketitle
\section{Introduction}
It is well known that the theory of strict partial orders and the theory of partial orders are ``equivalent'', i.e., they have the same essential content. To capture this vague idea, defining a precise equivalence between theories, several formal definitions have been made, e.g., \emph{logical equivalence}, \emph{definitional equivalence}, \emph{categorical equivalence}, etc. Which theory is equivalent to which other theory depends on the point of view from which one decides to explore the equivalence between the theories in question.

In the last few decades, the concept of equivalence between theories (henceforth: ``theory-equivalence'') has become important for studying the connections between formal theories. Many interesting results have been derived from investigating such equivalence, e.g., \cite{AMNmutual}, \cite{Glymour}, \cite{Japaridze}, \cite{pinter}  and \cite{Vis06}. We can also look at the question starting from non-equivalence. Given two non-equivalent theories (according to any chosen definition of theory-equivalence), some natural questions arise: (1) Can these theories be modified into equivalent theories (in a non-trivial way)? (2) If this can be done, can we do it in finitely many steps? In other words, what is the degree of their non-equivalence?

In this article, we lay down the first steps of a research programme to answer these questions. In order to investigate some ways to measure how far two theories are from each other; we introduce a framework that can give a qualitative and quantitative analysis of the connections between formal theories. We focus on formal theories that are formulated in any of the following logical systems: sentential logic, ordinary first order logic (FOL), finite variables fragments of FOL and/or infinitary versions of FOL. We develop several notions for distances between theories, we discuss these notions and we make comparisons between them.

The idea is very simple: based on a symmetric relation capturing a notion of minimal change, we introduce a general way to define a distance on any class of objects (not just theories) equipped with an equivalence relation. The idea is a generalization of the distance between any two nodes in the same graph, in graph theory. After, we give particular examples when the given class is a class of theories and the equivalence relation is a fixed notion of theory-equivalence.

The first particular example, is that of logical equivalence. As a measure for the degree of logical non-equivalence,  we introduce the concept of \emph{axiomatic distance}. The idea is to count the minimum number of axioms that are needed to be added or ``removed'' to get from one theory to the other.\footnote{By ``removing an axiom'' here we
only mean the trivial converse of adding an axiom in the following sense: $T$
is a theory resulting from ``removing'' one axiom from $T'$ if $T'$ can
be reached from $T$ by adding one axiom.} Since any finite number of axioms can be
concatenated by conjunction resulting in only one axiom, one may think
that the axiomatic distance, if it is finite, between two given theories
$T$ and $T'$ must be $\le 2$, i.e., we need at most two steps to get $T$
from $T'$: one step for axiom addition and another one for axiom removal. This is why we have the intuition that axiomatic distance may not be very interesting, cf., Problem~\ref{prob:axdist} and Theorem~\ref{thm:axdist}.

Then we turn to definitional equivalence. Two theories are definitionally equivalent if they cannot be distinguished by a concept (a formula defining some notion). As a measure for the degree of definitional non-equivalence, we define \emph{conceptual distance}. This distance counts the minimum number of \textit{concepts} that separate two theories. We find that this distance is of special interest in the study of logic. We give examples and we count conceptual distance between some specific theories, see, e.g., Theorem~\ref{thm:non-discrete} and Theorem~\ref{NSEM}. We also explore a connection between conceptual distance and \emph{spectrum of theories} which is a central topic in model theory, cf., Theorem~\ref{prop:spectrum}.

In algebraic logic, Lindenbaum-Tarski algebras of logical theories (sometimes these are called concept algebras) are often introduced as the algebras of different concepts of the corresponding theories. Thus, counting concepts amounts to counting elements of Lindenbaum-Tarski algebras. In fact, our definition of conceptual distance herein is a careful translation of an algebraic distance between Lindenbuam-Tarski algebras. Such an algebraic distance allows us to define conceptual distance between theories in any algebraizable logic, e.g., modal logic and intuitionistic logic. This general algebraic distance (and its application on concept algebras) is planned to be investigated in details in a forthcoming algebra oriented paper.

Furthermore, we investigate the possible application of conceptual distance in the logical foundation of physical theories in ordinary first order logic. We prove that conceptual distance between classical and relativistic kinematics is one. In other words, only one concept distinguishes classical and relativistic kinematics: the existence of a class of observers who are at absolute rest. This is indeed an interesting result in its own right, not only for logicians but also for physicists. Such a result opens several similar questions about how many concepts (and what are they) differentiate two physical theories when their phenomena can be described in FOL.

In philosophy of physics, this might be important because, on the one hand it is clear that we are not presently converging towards one unified theory of physics in the sense of converging to one set of laws from which all the phenomena of physics can be derived. On the other hand, we can give logical foundation to several physical theories: Newtonian mechanics, relativity theories and some parts of quantum theory. Given these logical representations, we would like to know the exact relationship between physical theories. If we know this, then we can form an impression of how far we are from such a philosophical dream -- the dream of the unity of physics. Or, we can adjust our hopes and expectations, and rest content with a unity of science at a more general level: as a network of logical theories with precise relations between them.

With definitions and metrics on distance developed here, we have maps of the network of logical theories. When we draw such maps of networks, the topology may suggest very interesting and fruitful questions. For instance: if there is a distance other than zero or one, then is there already a known theory in between? or if not, we can ask what are the limitative properties of that theory and what is its philosophical significance? By engaging in such studies, we see the ``edge'' of the limitative results, and by examining this edge we more precisely understand the rapport between meta-logical limitative results and physical phenomena.

In the present paper, we assume familiarity with the basic notions of set theory. For instance, what is a set, a class, a relation, etc. The only difference is that in several occurrences in this paper, we decided not to distinguish different kinds of infinities. Therefore, together with the standard notion of cardinality, we are going to speak about the \emph{\textbf{size} of set $X$}, defined as follows:
\begin{equation*}
  \vert\vert X\vert\vert\defeq \begin{cases}
    k & \text{$X$ is finite and has exactly $k$-many elements,}\\
    \infty & \text{if $X$ is an infinite set.}
  \end{cases}
\end{equation*}
We also make use of von Neumann ordinals. For example, $\omega$ is the smallest infinite ordinal, sometimes we denote $\omega$ by $\NN$ to indicate that it is the set of natural numbers (non-negative integers).
\section{Notions of logic} 
In the course of this paper, let $\alpha$ and $\beta\leq\alpha+1$ be two fixed ordinals. We consider a natural generalisation of ordinary first order logic, we denote it by $\Lg_{\alpha}^{\beta}$, which is inspired from the definitions and the discussions in \cite[section 4.3]{HMT85}. Roughly, the formulas of $\Lg_{\alpha}^{\beta}$ uses a fixed set of individual variables $\{v_i:i\in\alpha\}$ and relation symbols of rank strictly less than $\beta$. For simplicity, we assume that our languages do not contain any function symbols and/or constant symbols. 

In particular, $\Lg_0^1$ is sentential (propositional) logic, while $\Lg_{\omega}^{\omega}$ is ordinary first order logic. The so-called finite variables fragments of first order logic are the logics $\Lg_n^{n+1}$, for finite ordinals $n$'s. When $\alpha$ and $\beta$ are infinite, $\Lg_{\alpha}^{\beta}$ is called infinitary logic. Throughout, since $\alpha$ and $\beta$ are fixed, languages, theories, etc., are understood to be languages for $\Lg_{\alpha}^{\beta}$, theories in $\Lg_{\alpha}^{\beta}$, etc.

\subsection{The syntax of $\Lg_{\alpha}^{\beta}$}
More precisely, A \emph{\textbf{language} $\LL$ for $\Lg_{\alpha}^{\beta}$} is a set of relation symbols such that each relation symbol $R\in\LL$ is assigned a rank $\rank(R)<\beta$. Relation symbols of rank $0$ are called \emph{\textbf{sentential constants}}. To construct the formulas of language $\LL$, we also need some other symbols: equality ``$=$'' (we deal with quantifier logics with identity), brackets ``$($'' and ``$)$'', conjunction ``$\land$'', negation ``$\lnot$'' and the existential quantifier ``$\exists$''. We also use the necessary symbols to write sequences of variables $(v_{i_m}:m\in I)$, for any indexing set $I\subseteq \alpha$. \emph{The set of \textbf{formulas} $\Fm$ of $\LL$} is the smallest set that satisfies:
\begin{enumerate}
\item[(a)] $\Fm$ contains each \emph{\textbf{basic formula} of $\LL$}, where the basic formulas are the following two types of formulas:
\begin{enumerate}
\item[(i)] The equalities $v_i=v_j$, for any $i,j\in\alpha$.
\item[(ii)] $R(v_{i_m}:m<\rank(R))$, for any relation symbol $R$. 
\end{enumerate}
\item[(b)] $\Fm$ contains $\varphi\land\psi$, $\lnot\varphi$ and $\exists v_i\,\varphi$, for each $\varphi,\psi\in \Fm$.
\end{enumerate}
As usual, we use the following abbreviations.
\begin{itemize}
\item If $P$ is a sentential constant,  then we just write $P$ instead of $P()$.
\item If $R$ is a relation symbol of finite positive rank, say $k$,  then we write $R(v_{i_0},\ldots,v_{i_{k-1}})$ instead of $R(v_{i_m}:m<k)$.
\item We us disjunction, implication, equivalence and universal quantifier:
\begin{eqnarray*}
&\varphi\lor\psi\defeq\lnot(\lnot\varphi\land\lnot\psi) \hspace{0.5cm} & \varphi\rightarrow\psi\defeq\lnot(\varphi\land\lnot\psi)\\
&\varphi\leftrightarrow\psi\defeq(\varphi\rightarrow\psi)\land(\psi\rightarrow\varphi) \hspace{0.5cm} & \forall v_i\,\varphi\defeq\lnot(\exists v_i\, \lnot\varphi)
\end{eqnarray*}
\item We also use grouped conjunction and disjunction: Empty disjunction is defined to be $\varphi\land\lnot\varphi$ and empty conjunction is defined to be $\varphi\lor\lnot\varphi$ (for any arbitrary but fixed formula $\varphi\in\Fm$). Let $\varphi_0,\ldots,\varphi_m\in\Fm$, then
$$\bigvee_{0\leq i\leq m}\varphi_i\defeq\varphi_0\lor\cdots\lor\varphi_m \ \ \text{ and } \ \ \bigwedge_{0\leq i\leq m}\varphi_i\defeq\varphi_0\land\cdots\land\varphi_m.$$
\end{itemize}
\subsection{The semantics of $\Lg_{\alpha}^{\beta}$}
A \emph{\textbf{model} $\mathfrak{M}$ for language $\LL$} is a non-empty set $M$ enriched with operations $R^{\mathfrak{M}}\subseteq {M^{\rank(R)}}$, for each $R\in\LL$ (for a sentential constant $P$, $P^{\mathfrak{M}}\subseteq {M^0}=\{\emptyset\}$).\footnote{So the meaning $P^{\mathfrak{M}}$ of a sentential constant $P$ can be either \textsf{true} ($T=\{\emptyset\}$) or \textsf{false} ($F=\emptyset$).} An \emph{\textbf{assignment} in $\mathfrak{M}$} is a function $\tau$ that assigns for each variable an element of the set $M$. Let $\varphi\in \Fm$ be any formula. The satisfiability relation $\mathfrak{M},\tau\models\varphi$ is defined recursively as follows:
\begin{eqnarray*}
\mathfrak{M},\tau\models R(v_{i_m}:m<\rank(R)) &\text{iff}& (\tau(v_{i_m}):m<\rank(R))\in R^{\mathfrak{M}},\footnotemark\\
\mathfrak{M},\tau\models v_i=v_j &\text{iff}& \tau(v_i)=\tau(v_j),\\
\mathfrak{M},\tau\models\varphi\land\psi &\text{iff}& \mathfrak{M},\tau\models\varphi \text{ and } \mathfrak{M},\tau\models\psi,\\
\mathfrak{M},\tau\models\lnot\varphi &\text{iff}& \mathfrak{M},\tau\not\models\varphi,\\
\mathfrak{M},\tau\models\exists v_i\,\varphi &\text{iff}& \text{ there is }a\in M\text{ such that }\mathfrak{M},\tau[v_i\mapsto a]\models\varphi,
\end{eqnarray*}
\footnotetext{If $\rank(P)$ is $0$, then $(\tau(v_{i_m}):m<\rank(P))$ is the empty sequence $\emptyset$. Hence $\mathfrak{M},\tau\models P$ iff $P^{\mathfrak{M}}$ is \textsf{true}.}
where $\tau[v_i\mapsto a]$ is the assignment which agrees with $\tau$ on every variable except $\tau[v_i\mapsto a](v_i)=a$. The cardinality of $\mathfrak{M}$ is defined to be the cardinality of $M$. A formula $\varphi$ is said to be \emph{\textbf{true} in $\mathfrak{M}$}, in symbols $\mathfrak{M}\models\varphi$, iff $\mathfrak{M},\tau\models\varphi$, for every assignment $\tau$ in $\mathfrak{M}$. A formula $\varphi$ is said to be a \emph{\textbf{tautology}} iff it is true in every model for $\LL$. The \emph{\textbf{theory} of $\mathfrak{M}$} is defined as:
$$\Th(\mathfrak{M})\defeq\{\varphi\in \Fm:\mathfrak{M}\models\varphi\}.$$

We say that \emph{two models $\mathfrak{M}$ and $\mathfrak{N}$ for language $\LL$ are \textbf{isomorphic}} iff there is a bijection $f:M\rightarrow N$ between their underlying sets that respects the meaning of the relation symbols, i.e., for each $R\in\LL$, 
$$(a_i:i<\rank(R))\in R^{\mathfrak{M}}\iff (f(a_i):i<\rank(R))\in R^{\mathfrak{N}}.$$
\subsection{Theories in the logic $\Lg_{\alpha}^{\beta}$}
\begin{definition}
Suppose that $\LL$ is a language and let $\Fm$ be its set of formulas. A \emph{\textbf{theory} $T$ of $\LL$} is a set of formulas (subset of $\Fm$).
\end{definition}

We use the same superscripts and subscripts for theories and their corresponding languages and formulas. For example, if we write $T'$ is a theory, then we understand that $T'$ is a theory of language $\LL'$ whose set of formulas is $\Fm'$. A \emph{\textbf{model} for theory $T$} is a model for $\LL$ in which every $\psi\in T$ is true. We say that theory $T$ is \emph{\textbf{consistent}} iff there is at least one model for $T$.
\begin{definition}Let $T$ be a theory and let $\kappa$ be any cardinal. The \emph{\textbf{spectrum} of $T$}, in symbols  $I(T,\kappa)$,  is the number of its different models (up to isomorphism) of cardinality $\kappa$. This number is defined to be $\infty$ if $T$ has infinitely many non-isomorphic models of cardinality $\kappa$.
\end{definition}
We say that \emph{a fomrula $\varphi$ is a \textbf{theorem} of theory $T$}, in symbols $T\models\varphi$, iff $\varphi$ is true in every model for $T$. The \emph{set of \textbf{consequences} of theory $T$} is defined as follows:
$$\Cn(T)\defeq\{\varphi\in \Fm: T\models \varphi\}.$$

\begin{definition}
Two theories $T_1$ and $T_2$ are called \emph{\textbf{logically equivalent}}, in symbols $T_1\equiv T_2$, iff they have the same consequences, i.e., $\Cn(T_1)=\Cn(T_2)$.
\end{definition}

\subsection{More notions for theory-equivalence}
A \emph{\textbf{translation} of language $\LL_1$ into language $\LL_2$} is a map $\tr: \Fm_1\rightarrow \Fm_2$ such that the following are true for every $\varphi,\psi\in\Fm$ and every $v_i,v_j$.
\begin{itemize}
\item $\tr(v_i=v_j)$ is $v_i=v_j$.
\item $\tr$ commutes with the Boolean connectives: $$\tr(\neg\varphi)=\neg \tr(\varphi) \ \text{ and } \ \tr(\varphi\land\psi)=\tr(\varphi)\land \tr(\psi).$$
\item Finally, $\tr(\exists v_i \, \varphi)=\exists v_i \, \tr(\varphi)$.\footnote{
In the case of ordinary first order logic (when $\alpha=\beta=\omega$), to define a translation $\tr:\Fm_1\rightarrow\Fm_2$, it suffices to define $\tr$ on the basic formulas in $\Fm_1$ of the form $v_i=v_j$ and $R(v_0,\ldots, v_{m-1})$. Then, using Tarski's substitution observation, we can define
\begin{multline*}\tr(R(v_{i_1},\ldots,v_{i_{m}}))=\exists v_0 (v_0=y_1\land\cdots\land\exists v_{m-1}(v_{m-1}=y_{m}\land \\ \exists y_1(y_1=v_{i_1}\land\cdots\land\exists y_{m}(y_{m}=v_{i_{m}}\land \tr(R(v_0,\ldots,v_{m-1})))))),
\end{multline*}
where $y_i=v_{l+i}$ and $l$ is the maximum of $0,\ldots,m-1,i_1,\ldots,i_{m}$. This can be extended in a unique way to a translation that covers the whole $\Fm_1$.\label{footnote:Tarski}}
\end{itemize}

\begin{definition}\label{defeq}
Suppose that $T_1$ and $T_2$ are theories in languages $\LL_1$ and $\LL_2$, respectively, and $\tr$ is a translation of $\LL_1$ into $\LL_2$. The translation $\tr$ is said to be an \emph{\textbf{interpretation} of $T_1$ into $T_2$} iff it maps theorems of $T_1$ into theorems of $T_2$, i.e., for each formula $\varphi\in \Fm_1$,
$$T_1 \models \varphi \implies T_2\models \tr(\varphi).$$
\begin{enumerate}[(a)]
\item[(a)] An interpretation $\tr$ of  $T_1$ into $T_2$ is called \emph{\textbf{a faithful interpretation} of $T_1$ into $T_2$} iff  for each formula $\varphi\in \Fm_1$,
$$T_1 \models \varphi \iff T_2\models \tr(\varphi).$$ 
\item[(b)] An interpretation $\tr_{12}$ of $T_1$ into $T_2$ is called \emph{\textbf{a definitional equivalence} between $T_1$ and $T_2$} iff there is an interpretation $\tr_{21}$ of $T_2$ into $T_1$ such that  
\begin{itemize}
\item $T_1 \models \tr_{21}\big(\tr_{12}(\varphi)\big) \leftrightarrow \varphi$,
\item $T_2 \models \tr_{12}\big(\tr_{21}(\psi)\big) \leftrightarrow\psi$.
\end{itemize}
for every $\varphi\in \Fm_1$ and $\psi\in \Fm_2$. In this case, $\tr_{21}$ is also a definitional equivalence.
\end{enumerate}
\end{definition}
\begin{definition}
Two theories $T_1$ and $T_2$ are said to be \emph{\textbf{definitionally equivalent}}, in symbols $T_1 \intertrans T_2$, iff there is a definitional equivalence between them.
\end{definition}
In the literature, there are several ways to define definitional equivalence. Here, we use a variant of the definition in \cite[Definition 4.3.42 and Theorem 4.3.43]{HMT85}. For a discussion on the different definitions of definitional equivalence, see \cite{LSBH}, and we refer to \cite{Vis06} for a category theory based discussion.
\begin{prop}\label{prop:faithful}
Let $T_1$ and $T_2$ be two theories and suppose that $\tr_{12}:\Fm_1\rightarrow\Fm_2$ is a definitional equivalence between $T_1$ and $T_2$, then $\tr_{12}$ is also a faithful interpretation. 
\end{prop}
\begin{proof}Let $T_1$ and $T_2$ be two theories, and let $\tr_{12}:\Fm_1\rightarrow\Fm_2$ be a definitional equivalence between them. Let  $\varphi\in \Fm_1$, we should show that $T_1 \models \varphi$ iff $T_2\models \tr_{12}(\varphi).$ Since $\tr_{12}$ is an interpretation, we have that $T_1 \models \varphi$ implies $T_2\models \tr_{12}(\varphi)$.
To show the converse, let us assume that $T_2\models \tr_{12}(\varphi)$. By Definition~\ref{defeq}, there is an interpretation $\tr_{21}$ of $T_2$ into $T_1$ such that  $T_1 \models \tr_{21}\big(\tr_{12}(\varphi)\big) \leftrightarrow \varphi$. Since $\tr_{21}$ is an interpretation and $T_2\models \tr_{12}(\varphi)$, we have $T_1 \models \tr_{21}\big(\tr_{12}(\varphi)\big)$. Consequently, $T_1 \models \varphi$ since $T_1 \models \tr_{21}\big(\tr_{12}(\varphi)\big) \leftrightarrow \varphi$; and this is what we need to show.
\end{proof}

\begin{definition}\label{def-con-ext}
Let $T_1$ and $T_2$ be two theories. We say that \emph{$T_2$ is a \textbf{conservative extension} of $T_1$}, in symbols $T_1\sqsubseteq T_2$, iff $\Fm_1\subseteq \Fm_2$ and, for all $\varphi\in \Fm_1$, $T_2\models \varphi\iff T_1\models \varphi$.  
\end{definition}
We note that $T_1\sqsubseteq T_2$ iff the identity translation $id:\Fm_1\rightarrow \Fm_2$ is a faithful interpretation. It is also worth mentioning that 
$T_1\sqsubseteq T_2\iff T_1\equiv \Cn(T_2)\cap\Fm_1$.

\section{Cluster networks \& Step distance}\label{sd}
Now, we introduce a general way of defining a distance on any given class $\X$. We note that our target is to define distances on the class of all theories, thus we need to work with classes which are not necessarily sets. 

\begin{definition}
By a \textbf{cluster} $(\X,\E)$ we mean a class $\X$ equipped with an equivalence relation $\E$.\footnote{\label{classnote}All definitions in this section can be formulated within von Neumann--Bernays--G\"odel set theory (NBG). Of course, ordered pairs of proper classes cannot be formulated even in NBG, but we do not really need ordered pairs here. Our definitions can be understood as follows: ``for all classes $\X$, $\E$, etc., having certain properties there are classes $\dd$, etc., such that...''. We use the notations $(-,-)$ only to make our statements easier to be read and our proofs easier to be followed.}
\end{definition} 

We are interested in distances according to which some different objects are indistinguishable. Indeed, it is natural to treat equivalent theories as if they were of distance $0$ from each other. As we mentioned in the introduction, there are several notions of equivalence between theories. Such equivalence thus can be represented in the cluster of theoreis by the relation $\E$.

\begin{definition}
A \emph{\textbf{cluster network}} is a triple $(\X,\E,\Ss)$, where $(\X,\E)$ is
a cluster and $\Ss$ is a symmetric relation on $\X$.\footnote{See footnote~\ref{classnote} above.
}
\end{definition}

Given a cluster network $(\X,\E,\Ss)$. A \emph{\textbf{path} leading from $T\in \X$ to $T'\in \X$ in $(\X,\E,\Ss)$} is a finite sequence $b_1,\ldots, b_m$ of $0$'s and $1$'s such that there is a sequence $T_0,\ldots,T_m$ of members of $X$ with $T_0=T$, $T_m=T'$ and, for each $1\leq i\leq m$,
$$b_{i}=0\iff T_{i-1} \, \E \, T_{i} \ \ \text{ and } \ \ b_{i}=1\iff T_{i-1}\, \Ss\, T_{i}.$$ 
The \emph{\textbf{length} of this path} is defined to be $\sum_{i=1}^m b_i$. Two objects \emph{$T,T'\in \X$ are \textbf{connected} in $(\X,\E,\Ss)$} iff there is a path leading from one of them to the other in $(\X,\E,\Ss)$. 

\begin{definition}\label{def-step-distance}
Let $\mathcal{X}=(\X,\E,\Ss)$ be a cluster network. The \emph{\textbf{step distance} on $\mathcal{X}$} is the function $\dd_{\mathcal{X}}:\X\times \X\rightarrow \NN\cup\{\infty\}$ defined as follows. For each $T,T'\in \X$:
\begin{itemize}
\item If $T$ and $T'$ are not connected in $(\X,\E,\Ss)$, then $\dd_{\mathcal{X}}(T,T')\defeq\infty$.
\item If $T$ and $T'$ are connected in $(\X,\E,\Ss)$, then $$\dd_{\mathcal{X}}(T,T')\defeq\min\lbrace k\in\NN: \exists \text{ a path leading from $T$ to $T'$ whose length is $k$}\rbrace.$$
\end{itemize}
\end{definition}

The equivalence relation $\E$ represents pairs that cannot be
distinguished by the step distance, while the symmetric relation $\Ss$
represents the pairs of objects that are (at most) one step away from each other. The step distance then counts the minimum number of steps needed to reach an object starting from another one. We may need to say that infinitely many steps are needed, so we allow $\infty$ in the range of the step distance.

\begin{example}
Let $\X$ be any class, let $\E$ be the identity relation and let $\Ss=\X\times \X$. Then, $\mathcal{X}=(\X,\E,\Ss)$ is a cluster network and its step distance is the following discrete distance:
\begin{equation*}
\dd_{\mathcal{X}}(T,T')=\begin{cases}
0 & \text{ if }  T=T',\\
1 & \text{ if } T\not=T'.
\end{cases}
\end{equation*}
\end{example}

\begin{thm}\label{thm:min-path}
Let $\mathcal{X}=(\X,\E,\Ss)$ be a cluster network and let $\dd_{\mathcal{X}}:\X\times \X\rightarrow\NN\cup\{\infty\}$ be the step distance on $\mathcal{X}$. The following are true for each $T_1,T_2,T_3\in \X$:
 \begin{enumerate}
\item[(a)] $\dd_{\mathcal{X}}(T_1,T_2)\geq 0$, and $\dd_{\mathcal{X}}(T_1,T_2)=0\iff T_1 \, \E \, T_2$.
\item[(b)] $\dd_{\mathcal{X}}(T_1,T_2)=\dd_{\mathcal{X}}(T_2,T_1)$.
\item[(c)] $\dd_{\mathcal{X}}(T_1,T_2)\leq \dd_{\mathcal{X}}(T_1,T_3)+\dd_{\mathcal{X}}(T_3,T_2)$.
\end{enumerate}
\end{thm}
\begin{proof}    
Let $\mathcal{X}=(\X,\E,\Ss)$ and $\dd_{\mathcal{X}}$ be as required. Let $T_1,T_2,T_3\in \X$.
\begin{enumerate}
\item[(a)] Clearly, $\dd_{\mathcal{X}}(T_1,T_2)\ge 0$ for any two $T_1,T_2\in \X$, and $\dd_{\mathcal{X}}(T_1,T_2)=0$ if $T_1 \, \E \, T_1$ because then $0$ is a path from $T_1$ to $T_2$ in $\mathcal{X}$. If $\dd_{\mathcal{X}}(T_1,T_2)=0$, then there is a path $0,\ldots,0$ from $T_1$ to $T_2$ in $\mathcal{X}$. So there is a sequence $T'_0,\ldots,T'_m\in \X$ such that $T_1=T'_0$, $T_2=T'_m$ and $T'_{i-1} \, \E \, T'_{i}$ for each $1\leq i\leq m$. Hence, $T_1\, \E \, T_2$ since $\E$ is transitive.
\item[(b)] The symmetry is satisfied because $\E$ and $\Ss$ are symmetric relations. Hence, if $b_1,\ldots,b_m$ is a path leading from $T_1$ to $T_2$ in $\mathcal{X}$, then $b_m,\ldots,b_1$ is a path leading from $T_2$ to $T_1$ in $\mathcal{X}$.
\item[(c)] The triangle inequality $\dd_{\mathcal{X}}(T_1,T_2)\leq \dd_{\mathcal{X}}(T_1,T_3)+\dd_{\mathcal{X}}(T_3,T_2)$ follows from the definition because, if $b_1,\ldots,b_m$ is a path leading from $T_1$ to $T_3$ in $\mathcal{X}$ and $c_1,\ldots,c_k$ is a path leading from $T_3$ to $T_2$ in $\mathcal{X}$, then $b_1,\ldots,b_m,c_1,\ldots,c_k$ is a path leading from $T_1$ to $T_2$ in $\mathcal{X}$.\qedhere
\end{enumerate}
\end{proof}

\begin{remark}\label{rem}
Let $\mathcal{X}=(\X,\E,\Ss)$ and $\mathcal{X}'=(\X',\E',\Ss')$ be cluster networks such that $\X\subseteq \X'$, $\E\subseteq\E'\cap (\X\times \X)$, and $\Ss\subseteq\Ss'\cap (\X\times \X)$.  Since every path in $\mathcal{X}$ is contained in $\mathcal{X}'$, it is easy to see that $\dd_{\mathcal{X}}(T_1,T_2)\ge  \dd_{\mathcal {X}'}(T_1,T_2)$ for each $T_1,T_2\in \X$.
\end{remark}

Now, we use the above general settings to define distances between theories. Before we start, we need the following convention: Suppose that we are given two theories $T$ and $T'$. We write \label{implication-theories}$T\leftarrow T'$ iff there is $\varphi\in \Fm$ such that $T\cup\{\varphi\}\equiv T'$. We also write $T-T'$ iff either $T\leftarrow T'$ or $T'\leftarrow T$. Conventionally, we call the relation $\leftarrow$ \emph{\textbf{axiom adding}}, while the converse relation $\rightarrow$ is called \emph{\textbf{axiom removal}}. It is easy to see that the following are true for any theories $T_1,T_2$ and $T_3$.
\begin{eqnarray}
T_1\leftarrow T_2 \ \& \ T_2\leftarrow T_3 &\implies& T_1\leftarrow T_3,\label{lequiv}\\
T_1\equiv T_2 \ \& \ T_2\leftarrow T_3 &\implies& T_1\leftarrow T_3,\label{omittingequiv1}\\
T_1\leftarrow T_2 \ \& \ T_2\equiv T_3 &\implies& T_1\leftarrow T_3.\label{omittingequiv2}
\end{eqnarray}

\begin{definition} \label{def-axiomatic-distance}
Let $\X$ be a class of some theories (in the logic $\Lg_{\alpha}^{\beta}$) and consider the cluster network $(\X,\equiv,-)$. We call the step distance on this cluster network \emph{\textbf{axiomatic distance} on $\X$}. This step distance will be denoted by $\Ad_{\X}$.
\end{definition}

Let $\X$ be a class of theories. We note the following. If there is a path between $T,T'\in\X$ in the cluster network $(\X,\equiv,-)$, then both $T$ and $T'$ must be formulated in the  same language. In other words, if $T,T'$ are formulated on different languages, then $\Ad_{\X}(T,T')=\infty$. This is because two theories can be logically equivalent only if they are formulated on the same language. 

\begin{example}\label{ad-inc}Suppose that $\alpha\geq 1$ or $\beta\geq 1$. Let $\X$ be a class of theories. Let $T,T_{\bot}\in \X$ be two theories formulated in the same language. Suppose that $T$ is consistent while $T_{\bot}$ is inconsistent. Then, adding a contradiction to $T$ ensures that $\Ad_{\X}(T,T_{\bot})=1$. 
\end{example}

\begin{example}\label{ad-empty}
Let $\X$ be a class of theories. Let $T,\emptyset\in\X$ be two theories formulated in the same language such that $\emptyset$ is an empty theory (i.e., empty set of formulas). Suppose that $T$ is finitely axiomatizable, then we have either $$\Ad_{\X}(T,\emptyset)=1 \ \text{ or } \ T\equiv \emptyset.$$
\end{example}

Thus, in the class of all theories, the axiomatic distance between any two finitely axiomatizable theories is $\leq 2$.
\begin{example}Suppose that $\alpha\ge 3$ and $\beta\geq 3$. Let $\X$ be the set of all consistent theories of binary relations, let $T_P$ be the theory of partial orders, and let $T_E$ be the theory of equivalence relations. Then $\Ad_{\X}(T_P,T_E)=2$. Clearly, $\Ad_{\X}(T_P,T_E)\ge 2$ because none of $T_p$ or $T_E$ implies the other, and, by Example~\ref{ad-empty} and Theorem~\ref{thm:min-path} (c), $\Ad_{\X}(T_P,T_E)\le \Ad_{\X}(T_P,\emptyset) + \Ad_{\X}(\emptyset,T_E)=2$.
\end{example}
This gives us the intuition that, in most of the cases, the axiomatic distance is either $0$, $1$, $2$ or $\infty$. For example, in any class $\X$ containing only complete and consistent theories, $\Ad_{\X}$ is either $0$ or $\infty$ because adding an axiom to a complete theory is either results in an equivalent theory or in an inconsistent one. In the remaining of this section, we assume that $\alpha\geq 1$ or $\beta\geq 1$.

\begin{problem}\label{prob:axdist}
Let $\X$ be the class of all consistent theories in $\Lg_{\alpha}^{\beta}$. Is it true that, if the axiomatic distance between $T,T'\in\X$ is finite, then it must be $\leq 2$?
\end{problem}

Now, let us try to answer the above problem. We define the properties illustrated in Figure~\ref{fig:amalgamation}.  Let $\X$ be a class of theories. We say that \emph{$\X$ has the \textbf{theory amalgamation property}} iff for each $T\in\X$, if there are $T_1,T_2\in\X$ such that $T_1\rightarrow T\leftarrow T_2$, then there is $T'\in \X$ such that $T_1\leftarrow T'\rightarrow T_2$. Analogously, we say that $\X$ has the \emph{\textbf{theory co-amalgamation property}} iff for each $T\in\X$, if there are $T_1,T_2\in\X$ such that $T_1\leftarrow T\rightarrow T_2$, then there is $T'\in \X$ such that $T_1\rightarrow T'\leftarrow T_2$.

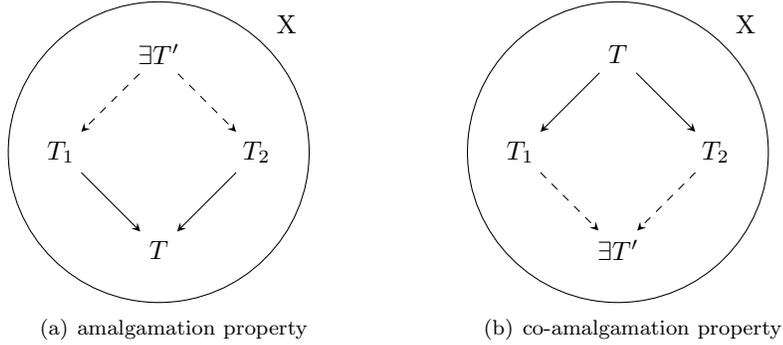
\begin{figure}[!ht]
  \centering
  \subfigure[amalgamation property]{
         \centering
    \begin{tikzpicture}[>=stealth]
      \node (T) at (0,-1.3) {$T$};
      \node (T1) at (1.3,0) {$T_2$};
      \node (T2) at (-1.3,0) {$T_1$};
      \node (T') at (0,1.3) {$\exists T'$};
      \draw[->] (T1) -- (T);
      \draw[->] (T2) -- (T);
      \draw[dashed,->] (T') -- (T1);
      \draw[dashed,->] (T') -- (T2);
      \draw (0,0) circle (2) ;
      \node[] at (1.7,1.7) {$\X$};
      \node at (2.4,0) {};
      \node at (-2,0) {};
    \end{tikzpicture}
}
\hspace{1cm}
\subfigure[co-amalgamation property]{
         \centering
    \begin{tikzpicture}[>=stealth]
      \node (T') at (0,-1.3) {$\exists T'$};
      \node (T1) at (1.3,0) {$T_2$};
      \node (T2) at (-1.3,0) {$T_1$};
      \node (T) at (0,1.3) {$T$};
      \draw[dashed,->] (T1) -- (T');
      \draw[dashed,->] (T2) -- (T');
      \draw[->] (T) -- (T1);
      \draw[->] (T) -- (T2);
      \draw (0,0) circle (2) ;
      \node[] at (1.7,1.7) {$\X$};
      \node at (2.4,0) {};
      \node at (-2,0) {};
    \end{tikzpicture}
}
\caption{Theory amalgamation properties}
\label{fig:amalgamation}
\end{figure}

\begin{thm}\label{thm:axdist} Let $\X$ be a class of theories having the theory amalgamation property or the theory co-amalgamation property. Then for all $T,T'\in \X$, we have the following
\begin{equation}\label{axdis}
    \Ad_{\X}(T,T')=
    \begin{cases}
      0 & \text{ if }  T\equiv T',\\
      1 & \text{ if $T'$ or $T$ is finitely axiomatizable over the other},\\
      \infty & \text{ if $T$ and $T'$ are not connected in $(\X,\equiv,-)$},\\ 
      2 & \text{ otherwise}.
    \end{cases}
\end{equation}
\end{thm}
\begin{proof}
Let us first assume that $\X$ has the theory amalgamation property. Suppose that $T,T'\in \X$ are connected via a path of length $3$ in the cluster network $(\X,\equiv,-)$. By \eqref{omittingequiv1} and \eqref{omittingequiv2}, we can find $T_1,T_2\in \X$ such that $T-T_1-T_2-T'$. Note that $T,T_1,T_2,T'$ have the same language $\LL$ and the same set of formulas $\Fm$. We first show that $T$ and $T'$ are connected by a path of length $2$. If at least two consecutive $-$ in the path $T-T_1-T_2-T'$ are in the same direction, e.g., $T\leftarrow T_1\rightarrow T_2\rightarrow T'$, then we are done by \eqref{lequiv}. So, we may assume that we have one of the cases illustrated in Figure~\ref{figcases}:

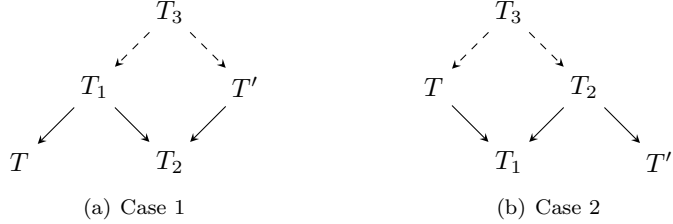
\begin{figure}[!ht]
\centering
\subfigure[Case 1]{
     \label{case 1}
    \begin{tikzpicture}[>=stealth]
\node (T) at (0,-0.5) {$T$};
\node (T1) at (1,0.5) {$T_1$};
\node (T2) at (2,-0.5) {$T_2$};
\node (T') at (3,0.5) {$T'$};
\node (T3) at (2,1.5) {$T_3$};
\draw[->] (T1) -- (T);
\draw[->] (T1) -- (T2);
\draw[->] (T') -- (T2);
\draw[->,dashed] (T3) -- (T1);
\draw[->,dashed] (T3) -- (T');
\end{tikzpicture}
}
\hspace{1.5cm}
\subfigure[Case 2]{
     \label{case 2}
    \begin{tikzpicture}[>=stealth]
\node (T) at (0,0.5) {$T$};
\node (T1) at (1,-0.5) {$T_1$};
\node (T3) at (1,1.5) {$T_3$};
\node (T') at (3,-0.5) {$T'$};
\node (T2) at (2,0.5) {$T_2$};
\draw[->] (T) -- (T1);
\draw[->,dashed] (T3) -- (T);
\draw[->,dashed] (T3) -- (T2);
\draw[->] (T2) -- (T1);
\draw[->] (T2) -- (T');
\end{tikzpicture}
}
\caption{$3$-paths can be replaced by $2$-paths}\label{figcases}
\end{figure}
\begin{enumerate}
\item[(a)] Suppose that we are in the first case $T\leftarrow T_1\rightarrow T_2\leftarrow T'$. Then, by the theory amalgamation property, there is $T_3\in\X$ such that $T_1\leftarrow T_3\rightarrow T'$. Hence, by \eqref{lequiv}, we have $T\leftarrow T_3\rightarrow T'$ which means $T$ and $T'$ are connected by a path of length $2$.
\item[(b)] Suppose that we are in the second case $T\rightarrow T_1\leftarrow T_2\rightarrow T'$. Then, by the theory amalgamation property, there is $T_3\in\X$ such that $T\leftarrow T_3\rightarrow T_2$. Hence, by \eqref{lequiv}, we have $T\leftarrow T_3\rightarrow T'$ which means $T$ and $T'$ are connected by a path of length $2$.
\end{enumerate}
Therefore, any path of length $3$ can be replaced by a path of length $2$. Now, we can prove the theorem. The first three cases of \eqref{axdis} are obvious, we need to show that otherwise the axiomatic distance is $2$. For this, it is enough to show that every path can be replaced by a path of length $2$. We use induction on the length of a path.
If the path is of length $l=3$, then we are done by the above
discussion. Suppose that we have already proven that every path not
longer than $l\ge 3$ can be replaced by a path of length $2$. Let path
$T-T_1-\cdots-T_{l}-T'$ be a path of length $l+1$. By induction
hypothesis, path $T-T_1-\cdots-T_{l}$ can be replaced by a path of
length $2$. Hence, path $T-T_1-\cdots-T_{l}-T'$ can be replaced by a
path of length $3$, which can be replaced by a path of length $2$ by the
induction hypothesis.

\begin{figure}[!ht]
\centering
\begin{tikzpicture}[scale=3,>=stealth]
\draw[->] (-0.1,2) -- node[above left]{$1$} (-0.9,1.2);
\draw[->] (0.1,2) --  node[above right]{$1$} (0.9,1.2) ;
\draw[-] (-0.8,1.2) -- (0,2);
\draw[-] (0.8,1.2) -- (0,2);
\draw[-] (-0.4,1.2) -- (0.2,1.8);
\draw[-] (0.4,1.2) -- (-0.2,1.8);
\draw[-] (-0,1.2) -- (0.4,1.6);
\draw[-] (0,1.2) -- (-0.4,1.6);
\draw[-] (0.4,1.2) -- (0.6,1.4);
\draw[-] (-0.4,1.2) -- (-0.6,1.4);
\draw[line width=2] (-0.8,1.2) -- (-0.6,1.4) -- (-0.4,1.2)--(-0.2,1.4)--(0,1.2)--
(0.2,1.4)--(0.4,1.2)--(0.6,1.4)--(0.8,1.2);
\node at (-0.9,1.1)[left]{$T$};
\node at (0.9,1.1)[right]{$T'$};
\end{tikzpicture}
\caption{Shortening paths using the theory amalgamation property.}\label{fig:pathshortening}
\end{figure}
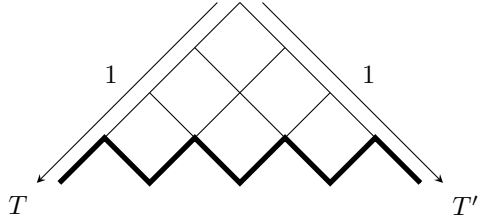

The above argument is illustrated in Figure~\ref{fig:pathshortening}. The proof of the case when $\X$ has the theory co-amalgamation property is completely analogous, but going downwards instead of upwards. 
\end{proof}

Even though Theorem~\ref{thm:axdist} shows the simplicity of axiomatic distance in plenty of cases, it leaves open the important case when $\X$ is the class of all consistent theories in $\Lg_{\alpha}^{\beta}$ (recall Problem~\ref{prob:axdist}).

\begin{prop}
Let $\X$ be the class of all consistent theories formulated over a fixed language $\LL$. Then, the class $\X$ does not have the theory amalgamation property. Moreover, e.g., if $\alpha=\omega$, $3\le\beta\le \omega$ and $\LL$ contains at least a binary relation symbol, then $\X$ does not have the theory co-amalgamation property either. 
\end{prop}
\begin{proof}
Let $\LL$ be a language and let $\X$ be the class of all consistent theories of language $\LL$. Let  $T=\emptyset$, $T_1=\{\exists v_0 \,(v_0=v_0)\}$ and $T_2=\{\forall v_0 \,(v_0\neq v_0)\}$ be three theories in $\X$. Then clearly, $T_1\rightarrow T\leftarrow T_2$, but there is no consistent theory $T'\in\X$ for which $T_1\leftarrow T'\rightarrow T_2$. Hence, $\X$ does not have the theory amalgamation property. 

Now assume that $\alpha=\omega$, $3\le \beta\le \omega$ and $\LL$ contains at least one binary relation symbol $R$. For every $n\in\NN$, let 
\begin{equation}\label{Psi-n}
\Psi^{(n)}\defeq \exists v_0\,\exists v_1\cdots \exists v_{n-1}\, \left( \left(\bigwedge_{0\leq i\neq j\leq n-1} v_i\neq v_j\right) \ \land \ \forall v_n\, \left(\bigvee_{0\leq i\leq n-1} v_n=v_i\right)\right).
\end{equation}
The formula $\Psi^{(n)}$ is saying that there are exactly $n$-many objects. Let
$$T_1=\{\neg \Psi^{(2n)}: n\in \NN\} \ \text{ and } \ T_2=\{\neg \Psi^{(2n+1)}: n\in \NN\}$$ be two theories in $\X$. Then a model of $T_1$ is a model for $\LL$ which have odd finite or infinite cardinality. Similarly, a model of $T_2$ is a model for $\LL$ which have even finite or infinite cardinality. Let $\varphi$ be a formula in language $\LL$ requiring that there are infinitely many objects. Using the relation symbol $R$ it is easy to write up such a formula by requiring that $R$ is irreflexive, transitive and serial. Then, clearly, $T_1\cup\{\varphi\}\equiv T_2\cup\{\varphi\}\equiv \{\varphi\}$ and $T_1\leftarrow \{\varphi\}\rightarrow T_2$. 

However, there is no theory $T'$ for which  $T_1\rightarrow T'\leftarrow T_2$ because (1) all the common consequences of $T_1$ and $T_2$ are tautologies, i.e., $\Cn(T_1)\cap \Cn(T_2)\equiv \emptyset$ since every model $\mathfrak{M}$ for $\LL$ is either a model for $T_1$ or a model for $T_2$, and (2) neither $T_1$ nor $T_2$ is finitely axiomatizable. The non-finite axiomatizability of $T_1$ can be confirmed as follows. For each $i\in\NN$, let $\mathfrak{M}_i$ be a model for $\LL$ whose cardinality is $2i$. Clearly, $\mathfrak{M}_i$ cannot be a model for $T_1$. However, the ultraproduct of the models $\mathfrak{M}_i$'s (over a free ultrafilter) is a model for $T_1$ because it is of infinite cardinality (for some details about ultraproducts, we refer the reader to \cite[Chapter 4.5]{modeltheory}). Thus, by \cite[Theorem 4.5.27 (ii)]{modeltheory}, it follows that $T_1$ is not finitely axiomatizable. Same applies to $T_2$ by taking ultraproduct of finite models having cardinalities $2i+1$ for all $i\in\NN$. 

Therefore, the theories $T_1$ and $T_2$ are not finitely axiomatizable over their common consequences. Hence, there is no theory $T'$ in $\X$ for which $T_1\rightarrow T'\leftarrow T_2$. Consequently, the class $\X$ does not have the theory co-amalgamation property.  
\end{proof}

It is worth noting that even if the range of the axiomatic distance consists only of four elements (cf., Theorem~\ref{thm:axdist} and Problem~\ref{prob:axdist}), this does not mean that this distance is not interesting. However, it might be true that the distance defined in this way does not tell us much information on the nature of the axioms separating two theories, adding any axiom is considered as one step. One can overcome this problem by giving weights to the axiom adding steps, e.g., considering the addition of certain kind of axioms as two or more steps. 

\section{Conceptual distance}\label{cd}
In 1935, A. Tarski introduced the so-called Lindenbaum-Tarski algebras as a device to establish correspondence between logic and algebras. The \emph{\textbf{Lindenbaum-Tarski algebra} of a theory $T$} is the quotient algebra obtained by factoring the algebra of formulas by the congruence relation of logical equivalence between formulas. These algebras are also called concept algebras, e.g., in \cite{andnem17}. This reflects the fact that the underlying set of a Lindenbaum-Tarski algebra consists of the different concepts that can be defined within the corresponding theory.

\begin{definition}
A \emph{\textbf{concept} in theory $T$} is a maximal set of logically equivalent formulas in $T$. In other words, a concept in $T$ is the set $[\varphi]_T\defeq \{\psi\in\Fm:T\models\varphi\leftrightarrow\psi\}$, for some formula $\varphi$.
\end{definition}

Intuitively, by a concept we mean a definition, no matter how many different ways one can write it equivalently. Now, we define conceptual size of a theory $T$ to be the size of the underlying set of the corresponding Lindenbaum-Tarski algebra. It might be more convenient in some cases to use cardinality here instead of size.

\begin{definition}\label{def:con-size}
Let $T$ be a theory and let $k\in\NN\cup\{\infty\}$. We say that the \emph{\textbf{conceptual size} of $T$ is $k$} and we write $\Cz(T)=k$ iff the set $\{[\varphi]_T:\varphi\in\Fm\}$ is of size $k$.%
\end{definition}

This is equivalent to saying that a theory $T$ is of conceptual size $k$ iff there is a maximal set $X\subseteq\Fm$ of size $k$ such that $T\not\models\varphi\leftrightarrow\psi$ for each $\varphi,\psi\in X$. It is also worthy of note that, when $\alpha\geq\omega$, a theory that has a model of at least two different elements cannot have a finite conceptual size. 

\begin{prop}\label{prop:czeq} Let $T,T'$ be two theories. If there is a faithful interpretation from $T$ to $T'$, then we have $\Cz(T)\le \Cz(T')$. 
Consequently,  
$$T\intertrans T'\implies \Cz(T)= \Cz(T').$$
\end{prop}
\begin{proof}
Let $T$ and $T'$ be two theories, and suppose that $\tr$ is a faithful interpretation of $T$ into $T'$. Now, there is a maximal set  $X\subseteq \Fm$ of size $\Cz(T)$ such that  $T\not\models\varphi\leftrightarrow\psi$ for any $\varphi,\psi\in X$. Let $$X'=\{\tr(\varphi):\varphi \in X\}\subseteq\Fm'.$$ Since $\tr$ is a faithful interpretation, $T'\not\models \tr(\varphi\leftrightarrow\psi)$. Hence $T'\not\models \tr(\varphi)\leftrightarrow \tr(\psi)$ since $\tr$ is an interpretation. Consequently, $\Cz(T)\le \Cz(T')$. If $T\intertrans T'$, then by Proposition~\ref{prop:faithful} there are faithful interpretations between $T$ and $T'$ in both directions. Hence $\Cz(T)\le \Cz(T')\le \Cz(T)$. Therefore, $\Cz(T)= \Cz(T')$ as desired.
\end{proof}

As a matter of fact, two theories are definitionally equivalent iff their Lindenbaum-Tarski algebras are isomorphic \cite[Theorem 4.3.43]{HMT85}. So, the above Proposition is a consequence of the straightforward observation: Two algebras of different sizes cannot be isomorphic. 

Now, we want to define a distance counting the minimum number of concepts that distinguish two theories $T$ and $T'$. It is very natural to explain the idea within the framework of Lindenbaum-Tarski algebras, since these are concept algebras and we want to count cocnepts. For simplicity, let us assume that the Lindenbaum-Tarski algebra $\mathfrak{A}$ of $T$ is embeddable into the Lindenbuam-Tarski algebra $\mathfrak{A}'$ of $T'$. Thus, the minimum number of concepts distinguish the two theories is equal to the minimum number of elements of $\mathfrak{A}'$ that we can add to $\mathfrak{A}$ (more precisely, to one of its copies inside $\mathfrak{A}'$) to generate the algebra $\mathfrak{A}'$, see Figure~\ref{algebra}.

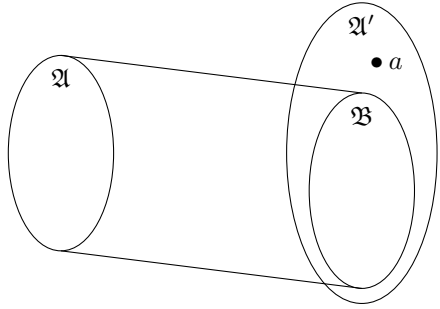
\begin{figure}[!htb]
\begin{tikzpicture}[scale=1]
    \draw (0,0) ellipse (0.7cm and 1.3cm);
    \draw (4,-0.5) ellipse (0.7cm and 1.3cm);
    \draw (4,0) ellipse (1cm and 2cm);
    \draw (4,0.8)--(0,1.3);
    \draw (4,-1.8) -- (0,-1.3);
    \node (A) at (0,1) {$\mathfrak{A}$};
    \node (B) at (4,0.5) {$\mathfrak{B}$};
    \node (A') at (4,1.7) {$\mathfrak{A}'$};
    \node (E) at (4.2,1.2) {$\bullet$};
    \node (a) at (4.45,1.2) {$a$};
\end{tikzpicture}
\caption{The distance between $\mathfrak{A}$ and $\mathfrak{A}'$ is one if $\langle B,a\rangle \cong \mathfrak{A}'$.}\label{algebra}
\end{figure}

The following definitions illustrate the above idea, but in terms of logic instead of algebras. As we mentioned before, the correspondence between the following definitions and the above algebraic idea is planned to be discussed in details in a forthcoming algebra oriented paper.

\begin{definition}\label{def:conpext}
We say that \emph{theory $T'$ is a \textbf{one-concept-extension} of theory $T$} and we write $T\leadsto T'$ iff $\LL'=\LL\cup\{R\}$, for some relation symbol $R$, and $T\sqsubseteq T'$. We also write $T\sim T'$ iff $T\leadsto T'$ or $T'\leadsto T$, and in this case we say that \emph{$T$ and $T'$ are \textbf{separated by at most one concept}}.
\end{definition}

Again, we understand concept removal to be the converse of concept adding. Later, in Section~\ref{Other distances}, we are going to introduce another notion for concept removal, and a corresponding distance notion.

\begin{definition}\label{conceptual distance}Let $\X$ be a class of theories. The step distance induced by the cluster network $(\X,\intertrans,\sim)$ is called \emph{\textbf{conceptual distance} on $\X$} and is denoted by
  $\Cd_{\cl{X}}$. In the case when $\X$ is the class of all theories in $\Lg_{\alpha}^{\beta}$, we denote the conceptual distance on $\X$ by $\Cd_{\alpha}^{\beta}$.
\end{definition}

By Remark~\ref{rem}, it is clear that $\Cd_{\alpha}^{\beta}(T,T')\geq \Cd_{\alpha}^{\gamma}(T,T')$ for any ordinal $\beta\leq\gamma\leq\alpha+1$ and any theories $T$ and $T'$ in $\Lg_{\alpha}^{\beta}$. It is also apparent that an inconsistent theory is of an infinite conceptual distance from any consistent theory, because relations $\intertrans$ and $\leadsto$ cannot make a consistent theory inconsistent and also cannot make an inconsistent theory consistent. Now, we give more examples.

\begin{lem}\label{lem:no-of-models}
Let $k,m$ be two finite numbers. Let $T$ and $T'$ be two theories such that $T\sqsubseteq T'$. Suppose that $\LL'=\LL\cup\{R\}$, for some relation $R$ whose rank $\rank(R)=m$. Then, $$I(T',k)\leq {2^{k^m}} I(T,k).$$
\end{lem}
\begin{proof}
This follows immediately from the following fact: If $\mathfrak{M}$ is a model of size $k$, then there are at most ${2^{k^m}}$ different ways of defining a relation of rank $m$ in $\mathfrak{M}$.
\end{proof}

\begin{thm}\label{thm:non-discrete} Suppose that $\beta\geq 1$. For every $n\in\NN\cup\{\infty\}$, there are theories $T$ and $T'$ in $\Lg_{\alpha}^{\beta}$ such that $\Cd_{\alpha}^{\beta}(T,T')=n$.
\end{thm}
\begin{proof}
Let $\mathcal{L}_{\infty}$ be a language for $\Lg_{\alpha}^{\beta}$ that consists of infinitely many relation symbols (describing infinitely many different concepts), each of which is of a fixed rank $\gamma<\beta$.
\begin{equation}\label{list}
R_1,R_2,\ldots
\end{equation}
For each $n\in\NN$, we let $\mathcal{L}_{n}\subseteq\mathcal{L}_{\infty}$ be the language consists of the first $n$-many relation symbols from the list in \eqref{list}, and we let $T_n^{\star}=\emptyset$ be the empty theory on language $\mathcal{L}_n$. Let $n\in\NN\cup\{\infty\}$. It is clear that $\Cd_{\alpha}^{\beta}(T_0^{\star},T_n^{\star})\leq n$. We need to prove the other direction. We consider the following cases:

Suppose that $\beta$ is finite. In this case, we assume that $\gamma=\beta-1$. Let $k\geq 1$ be a finite number. By Lemma~\ref{lem:no-of-models}, it
follows that, for each pair of theories $T_1,T_2$ in logic $\Lg_{\alpha}^{\beta}$,
\begin{equation}\label{beta-1}
T_1\leadsto T_2\implies I(T_2,k)\leq 2^{k^{\beta-1}} I(T_1,k).
\end{equation}
Moreover, it is easy to see that
\begin{equation}\label{modelsstar}
I(T_n^{\star},k)={2^{nk^{\beta-1}}} \ \ \text{(we define this
entity to be $\infty$ if $n=\infty$)},
\end{equation}
because there are exactly $n$-many relation symbols, each of which is of rank $\beta-1$, in $\LL_n$. Thus, at least
$n$-many steps are needed to increase  $I(T_0^{\star},k)=2^0=1$ to
$I(T_n^{\star},k)={2^{nk^{\beta-1}}}$. Therefore, we have $\Cd_{\alpha}^{\beta}(T_0^{\star},T_n^{\star})=n$ as
required.

\begin{figure}[htb]
    \centering
    \begin{tikzpicture}[scale=1]

\def\n{6}

\draw[fill, gray!21] (0,0) to (\n,0) to (\n,\n);

\draw[dotted,thick] (0,0) to (1,1) to (2,1.5) to (3,2.25) to (4,2.75) to (5,3.5) to (6,4);

\draw[gray] (0,0) node [black,left] {$2^{0} \ \ $};
\draw[gray] (0,1) node [black,left] {$2^{k^{\beta-1}}$} to (\n,1);
\draw[gray] (1,0) to (1,\n);
\draw[fill] (1,1) node[below right]{\scriptsize $T_{1}^{\star}$} circle [radius=.05];
\draw[fill] (0,0) node[below right]{\scriptsize $T_{0}^{\star}$} circle [radius=.05];

\foreach \x in {2,...,\n}
{
\draw[gray] (0,\x) node [black,left] {$2^{\x k^{\beta-1}}$}to (\n,\x);
\draw[gray] (\x,0) to (\x,\n);
\draw[fill] (\x,\x) node[below right]{\scriptsize $T_{\x}^{\star}$} circle [radius=.05];
}

\draw (0,1+\n) node[left]{$I(T,k)$} -- (0,0) -- (\n+0.5,0) node[below]{theories};

\end{tikzpicture}
\caption{Theorem~\ref{thm:non-discrete}: $\Cd_{\alpha}^{\beta}(T_0^{\star},T_n^{\star})=n$}
\end{figure}
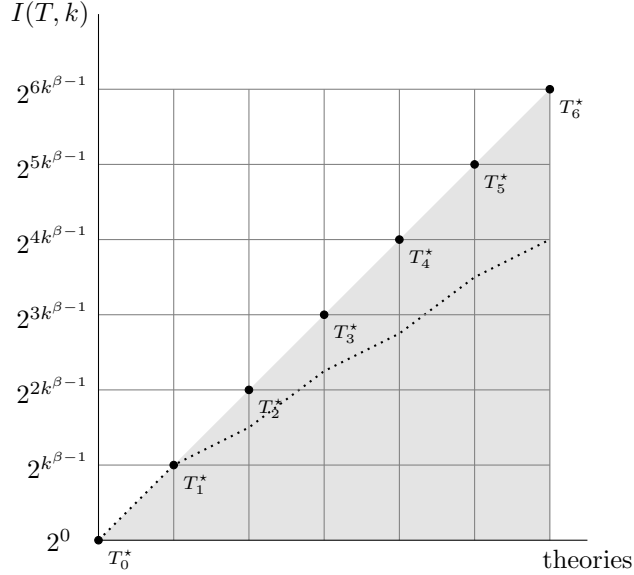

Suppose that $\beta$ is infinite. We do not need any extra assumption on $\gamma$. All what we need here is to count models of size $1$. By \eqref{modelsstar}, we have $I(T_n^{\star},1)=2^n$. Moreover, for any two theories $T_1$ and $T_2$, 
\begin{equation}
T_1\leadsto T_2\implies I(T_2,1)\leq 2 I(T_1,1).
\end{equation}
This is true because in a model of size $1$ there are at most two relations (of any fixed rank). Again, for the same reasons, we need at least $n$-many steps to increase  $I(T_0^{\star},1)=2^0=1$ to $I(T_n^{\star},1)=2^n$. Therefore, $\Cd_{\alpha}^{\beta}(T_0^{\star},T_n^{\star})=n$ as desired.
\end{proof}

Our constructions in the above proof for the case when $\beta$ is infinite (and hence $\alpha$ is also infinite) all have models of size $1$. In Theorem~\ref{NSEM} below, we investigate what happens if the theories in question do not have models of size $1$. First, we need the following lemma.

\begin{lem}\label{lem}Suppose that $\alpha=\beta=\omega$. Let $T_1$, $T_2$ and $T_3$ be theories such that
$$I(T_1,1)=I(T_2,1)=I(T_3,1)=0.$$ Then, if $T_1\leadsto T_2\leadsto T_3$, then there is a theory $T$ such that $T_1\leadsto T \intertrans T_3$.
\end{lem}
\begin{proof}
Suppose $T_1$, $T_2$ and $T_3$ are as required in the statement of the lemma above, and assume that $T_1\leadsto
T_2\leadsto T_3$. Then $\LL_3=\LL_1\cup\{R,S\}$ for some relation symbols $R$ and
$S$. Suppose that $\rank(R)=n$ and $\rank(S)=m$. Let $l=\max\{n,m\}+2$. Let
$\LL\defeq \LL_1\cup \{B\}$, for some new relation symbol $B$ of rank $l$. Every
model $\mathfrak{M}$ for $\LL_3$ can be extended to a model $\mathfrak{M}^+$
for $\LL$ by defining $B^{\mathfrak{M}^+}$ as follows:
$$B^{\mathfrak{M}^+}\defeq\{(a_0,\ldots,a_{l-1})\in{M^l}:\exists\text{ an assignment }\tau\big(\tau(v_0)=a_0,\cdots,\tau(v_{l-1})=a_{l-1}\text{ and }\mathfrak{M},\tau\models\psi\big)\},$$ 
where $\psi(v_0,\ldots, v_{l-1})\defeq \big(R(v_0,\ldots, v_{n-1})\land
v_{l-2}=v_{l-1}\big)\lor \big(S(v_0,\ldots,v_{m-1})\land v_{l-2}\neq
v_{l-1}\big)$. Let 
\begin{equation*}
T\defeq \{\varphi\in \Fm : \mathfrak{M}^+\models \varphi, \text{ for every
  model } \mathfrak{M} \text{ for } T_3\}.
\end{equation*}
We need to prove that $T_1\leadsto T$ and $T\intertrans T_3$. To prove that
$T_1\leadsto T$, it is enough to show that $T_1\sqsubseteq T$ (because
$\LL=\LL_1\cup \{B\}$). Let $\varphi \in \Fm_1$. We have 
\begin{equation*}
T_1\models \varphi \iff T_2\models \varphi \iff T_3\models \varphi \iff
T\models \varphi,
\end{equation*}
where the first two equivalences follow by the assumption $T_1\sqsubseteq
T_2\sqsubseteq T_3$, and the last equivalence follows by the definition of
$T$.  To show that $T\intertrans T_3$, we define the translations $\tr:
Fm\to Fm_3$ and $\tr': Fm_3\to Fm$ as follows:
$$\tr: B(v_0,\ldots,v_{l-2},v_{l-1})\mapsto \psi(v_0,\ldots, v_{l-1})\quad  \text{ and} $$
\begin{eqnarray*}
\tr': R(v_0,\ldots,v_{n-1})&\mapsto&  \exists v_{l-1}\, \big(B(v_0,\ldots,v_{l-2},v_{l-1})\land
(v_{l-2}=v_{l-1})\big)\\
\tr': S(v_0,\ldots,v_{m-1})&\mapsto&  \exists v_{l-1}\, \big(B(v_0,\ldots,v_{l-2},v_{l-1})\land
(v_{l-2}\neq v_{l-1})\big)
\end{eqnarray*}

We have defined $\tr$ and $\tr'$ on specific basic formulas and these can be
extended in a unique way to their domains, see
Footnote~\ref{footnote:Tarski}. Let $\mathfrak{M}$ be a model for $T_3$. By definition of the extension
$\mathfrak{M}^+$, 
\begin{equation}\label{def:m}
\mathfrak{M}^+\models B(v_0,\ldots,v_{l-2},v_{l-1}) \leftrightarrow \psi(v_0,\ldots,v_{l-2},v_{l-1}).
\end{equation}
Thus, by \eqref{def:m}, we have 
\begin{eqnarray*}
\mathfrak{M}^+\models \tr'(S(v_0,\ldots,v_{m-1})) &\leftrightarrow& \exists v_{l-1}\,
\big(B(v_0,\ldots,v_{l-2},v_{l-1})\land v_{l-2}\neq v_{l-1}\big)\\
&\leftrightarrow& \exists v_{l-1}\,\big(((R(v_0,\ldots,v_{n-1})\land
v_{l-2}=v_{l-1})\land \\ 
&& (S(v_0,\ldots,v_{m-1})\land v_{l-2}\neq v_{l-1}))\land v_{l-2}\neq
v_{l-1}\big)\\
&\leftrightarrow& \exists v_{l-1}\, (S(v_0,\ldots,v_{m-1})\land v_{l-2}\neq
v_{l-1}))\\
&\leftrightarrow& S(v_0,\ldots,v_{m-1}).
\end{eqnarray*}
The last $\leftrightarrow$ follows by the assumption that the cardinality of
$\mathfrak{M}$ is at least $2$ (and hence the same is true for $\mathfrak{M}^+$). Similarly, $\mathfrak{M}^+\models
\tr'(R(v_0,\ldots,v_{n-1}))\leftrightarrow R(v_0,\ldots,v_{n-1})$. 
Therefore, by the fact that $\tr$ and $\tr'$ are translations, it follows that 
\begin{equation}\label{mplus}
\mathfrak{M}^+\models \varphi \leftrightarrow \tr(\varphi)\ \text{ and }  \mathfrak{M}^+\models \psi \leftrightarrow \tr'(\psi)
\end{equation}
for all $\varphi\in\Fm$ and $\psi\in \Fm_3$.
By \eqref{mplus}, it is not hard to see that $\tr$ and $\tr'$ are definitional
equivalences, and the desired follows. 
\end{proof}
The above lemma is a direct consequence of the following elementary fact. In $\Lg_{\omega}^{\omega}$ (under some conditions), for any two relations $R$ and $S$, there is a relation $M$ such that $M$ is definable in terms of $R$ and $S$ and, conversely, both $R$ and $S$ are definable in terms of $M$, see \cite{goodman}. The idea of the above proof is distilled from \cite[Theorem 2.3.22]{hmt71}. 
\begin{cor}\label{cor}
Suppose that $\alpha=\beta=\omega$. Let $T_1,T_2,\ldots,T_n$, for some $n\geq 2$, be theories such that $I(T_i,1)=0$, for each $1\leq i\leq n$. 
Then, 
$$T_1\leadsto T_2 \leadsto \cdots \leadsto T_n\implies \text{ there is a theory } T  \text{ such that } T_1\leadsto T \intertrans T_n.$$
\end{cor}
\begin{proof}
This can be proved by a simple induction on $n$. If $n=2$, then we are obviously done. Suppose that $n\geq 3$ and $T_1\leadsto T_2 \leadsto \cdots T_{n-1}\leadsto T_{n}$. If by induction hypothesis we can assume that there is $T'$ such that $T_2\leadsto T'\intertrans T_n$, then $T_1\leadsto T_2\leadsto T'\intertrans T_n$. Therefore, by the Lemma~\ref{lem}, there is theory $T$ such that $T_1\leadsto T\intertrans T_n$.
\end{proof}
\begin{thm}\label{NSEM}
Suppose that $\alpha=\beta=\omega$. Let $T$ and $T'$ be theories on finite languages. Then, 
$$I(T,1)=I(T',1)=0\implies \Cd_{\omega}^{\omega}(T,T')=\infty \text{ or } \Cd_{\omega}^{\omega}(T,T')\leq 2.$$
\end{thm}
\begin{proof}
Let $T$ and $T'$ be theories on finite languages and suppose that $I(T,1)=I(T',1)=0$. Let us assume that $\Cd_{\omega}^{\omega}(T_1,T_2)<\infty$. Let $\LL$ be the empty language. For each $\varphi\in\Fm$,
\begin{equation}\label{reduct}
T_1\models\varphi\iff T_2\models \varphi.
\end{equation}
This is true because $\Cd_{\omega}^{\omega}(T,T')<\infty$ and the validity of each $\varphi\in\Fm$ is preserved under conservative extensions and definitional equivalences. Let $T=\{\varphi\in\Fm:T_1\models\varphi\text{ and }T_2\models\varphi\}$. Then, by \eqref{reduct}, it is true that $T\sqsubseteq T_1$ and $T\sqsubseteq T_2$. We also claim that
\begin{equation}\label{finish}
\Cd_{\omega}^{\omega}(T,T_1)\leq 1 \ \text{ and } \ \Cd_{\omega}^{\omega}(T,T_2)\leq 1.
\end{equation}
To show \eqref{finish}, let us assume that $\LL=\{R_i:i<m\}$, for some finite $m$. If $m=0$, then $T\equiv T_1$ and thus $\Cd_{\omega}^{\omega}(T,T_1)=0$. Assume that $m\not=0$. Let $\LL_0^{\star}=\{R_0\},\ldots, \LL_{m-1}^{\star}=\{R_0,\ldots,R_{m-1}\}$ and, for each $0\leq i\leq m-1$, $T_i^{\star}=\{\varphi\in\Fm_i^{\star}:T_1\models\varphi\}$. Hence, 
$T\leadsto T_0^{\star}\leadsto\cdots\leadsto T_{m-1}^{\star}\equiv T_1$. Clearly, for each $0\leq i\leq m-1$, $I(T_i^{\star},1)=0$ because $\lnot\Psi^{(1)}=\lnot\Big(\exists v_0\,\forall v_1\,(v_0=v_1)\Big)$ is a theorem of $T$, and hence is a theorem of $T_i^{\star}$. Thus, by Corollary~\ref{cor}, it follows that $\Cd_{\omega}^{\omega}(T,T_1)\leq 1$. Similarly, one can show that $\Cd_{\omega}^{\omega}(T,T_2)\leq 1$. Therefore, $\Cd_{\omega}^{\omega}(T_1,T_2)\leq \Cd_{\omega}^{\omega}(T,T_1)+\Cd_{\omega}^{\omega}(T,T_2)\leq 2$.
\end{proof}
We prove one more statement investigating the connection between the spectrum of theories and conceptual distance in $\Lg_{\omega}^{\omega}$.
\begin{thm}\label{prop:spectrum}Suppose that $\alpha=\beta=\omega$. Let $T$ and $T'$ be two theories formulated in countable languages. Then,
\begin{equation}\label{speceq}
    \Cd_{\omega}^{\omega}(T,T')<\infty \implies (\forall\text{ cardinal }\kappa) \,\big[ I(T,\kappa)\neq 0 \iff I(T',\kappa)\neq 0\big].
  \end{equation}
If $T$ and $T'$ are formulated in finite languages, then the converse of \eqref{speceq} is also true.
\end{thm}
\begin{proof}In this proof, we will make use of the formulas $\Psi^{(n)}$'s defined in \eqref{Psi-n}. Recall that with the formula $\Psi^{(n)}$, we require every model to be of cardinality $n$. Let $T$ and $T'$ be as required, and let $\kappa$ be any cardinal. Suppose that $\kappa$ is finite, then $T$ does not have a model of size $\kappa$ iff $T\models \neg\Psi^{(\kappa)}$. But, $T'\models\neg\Psi^{(\kappa)}\iff T\models \neg\Psi^{(\kappa)}$ because of the assumption $\Cd_{\omega}^{\omega}(T,T')<\infty$ and the fact that the validity of $\Psi^{(\kappa)}$ is preserved under conservative extensions and definitional equivalences. Hence, $T'$ has a model of cardinality $\kappa$ iff $T$ has a model of cardinality $\kappa$. If $\kappa$ is an infinite cardinality, then by L\"ovenheim--Skolem Theorem, $T$ and $T'$ have models of cardinality $\kappa$ iff they have infinite models. $T$ has an infinite model iff $T\models \exists v_0\, \exists v_1 \cdots \exists v_{n-1}\,\left ( \bigwedge_{0\leq i\neq j\leq n-1} v_i\neq v_j\right)$ for all $n\in \NN$. 
Again, the validity of the formulas $\exists v_0\, \exists v_1 \cdots \exists v_{n-1}\,\left ( \bigwedge_{0\leq i\neq j\leq n-1} v_i\neq v_j\right)$ is preserved under conservative extensions and definitional equivalences. Thus, the assumption $\Cd_{\omega}^{\omega}(T,T')<\infty$ implies that
$$T'\models \exists v_0\, \exists v_1 \cdots \exists v_{n-1}\,\left ( \bigwedge_{0\leq i\neq j\leq n-1} v_i\neq v_j\right)\iff T\models \exists v_0\, \exists v_1 \cdots \exists v_{n-1}\,\left ( \bigwedge_{0\leq i\neq j\leq n-1} v_i\neq v_j\right).$$ 
Hence, $T$ has an infinite model and thus a model of cardinality $\kappa$ iff $T'$ has such a model. 

To prove the converse of \eqref{speceq}, let us assume that $T$ and $T'$ are formulated in finite languages $\LL$ and $\LL'$, respectively. Assume that, for every cardinal $\kappa$, $T$ has a model of cardinality $\kappa$ iff $T'$ has a model of the same cardinality. Let $\LL_{0}$ be the empty language. Then, for every $\varphi\in Fm_0$,
$$T\models\varphi\iff T'\models\varphi.$$
Let $T_0=\{\varphi\in Fm_0: T\models\varphi \text{ and }T'\models\varphi\}$. Thus, $T_0\sqsubseteq T$ and $T_0\sqsubseteq T'$. Now, we can add the whole $\LL$ to $\LL_0$ in finitely many steps, because there only finitely many relation symbols in $\LL$, thus $\Cd_{\omega}^{\omega}(T_0,T)<\infty$. Similarly, $\Cd_{\omega}^{\omega}(T_0,T')<\infty$. Therefore, 
\begin{equation*}
\Cd_{\omega}^{\omega}(T,T')\leq \Cd_{\omega}^{\omega}(T_0,T)+\Cd_{\omega}^{\omega}(T_0,T')<\infty.\qedhere
\end{equation*}
\end{proof}

\begin{cor}
The conceptual distance between the theories of any two finite models of different cardinalities is infinite. More precisely, if $\mathfrak{A}$ and $\mathfrak{B}$ are two finite models of different cardinality, then $\Cd_{\omega}^{\omega}\big(\Th(\mathfrak{A}),\Th(\mathfrak{B})\big)=\infty$.
\end{cor}

For instance, given two cyclic groups $\langle k_1\rangle$ and $\langle k_2\rangle$ of orders $5$ and $7$, respectively, the conceptual distance between the theories of these groups is $\infty$. This might seem strange; these theories are about similar structures. But if we look carefully at the statement of the above corollary, we will find that it talks about theories of structures, not structures themselves. In other words, the conceptual distance between the theories of $\langle k_1\rangle$ and $\langle k_2\rangle$ cannot be granted as a distance between these two groups as algebraic structures. This conceptual distance can be rather considered as a distance between the Lindenbaum-Tarski algebras of the theories of these groups, which are of course of different nature than the groups themselves. 

\begin{cor}
There are infinitely many theories that are, in terms of conceptual distance, infinitely far from each other in $\Lg_{\omega}^{\omega}$.  
\end{cor}

\begin{problem}Let $\X$ be the class of all complete and consistent theories in $\Lg_{\omega}^{\omega}$, and let $T_1,T_2\in \X$. Is it always true that $$\Cd_{\omega}^{\omega}(T_1,T_2)=\Cd_{\X}(T_1,T_2)?$$
\end{problem}

\section{Conceptual distance in physics}\label{applications}
Each physical theory is established based on some preliminary decisions. These decisions are suggested by the accumulation and the assimilation of new knowledge. The methods used to improve physical theories are intuitively conceived and applied in a fruitful way, but many obvious ambiguities have appeared. To clarify these ambiguities, it was critical to introduce the {\emph{logical foundation of physical theories}}.

Even today the logic based axiomatic foundation of physical theories is intensively investigated by several research groups. For example, the Andr\'eka--N\'emeti school axiomatizes and investigates special and general relativity theories within ordinary first order logic, see, e.g., \cite{BigBook}, \cite{AMNBerlin} and \cite{Synthese}. For similar approaches related to other physical theories, see, e.g., \cite{Baltag2005} and \cite{krause2017}.

Following Andr\'eka--N\'emeti school's traditions, two theories \ax{ClassicalKin} and \ax{SpecRel} are formulated in ordinary first order logic $\Lg_{\omega}^{\omega}$ to capture the intrinsic structures of classical and relativistic kinematics. For the precise definitions of these theories, one can see \cite[p.67 and p. 69]{ClassRelKin}. In this section, we will investigate the conceptual distance between these two theories.

In \cite{diss} and \cite{ClassRelKin}, it was shown that these two theories can be turned definitionally equivalent by the following two concept manipulating steps: 
\begin{enumerate}
\item[(1)] adding the concept of an observer ``being stationary''  to the theory of relativistic kinematics \ax{SpecRel}, and 
\item[(2)] removing the concept of observers ``not moving slower than light'' from the theory of classical kinematics \ax{ClassicalKin}.
\end{enumerate}
Then, it was shown that even if observers ``not moving slower than light'' are removed from \ax{ClassicalKin} the resulting theory remains definitionally equivalent to \ax{ClassicalKin} and hence adding only the concept of ``being stationary'' to \ax{SpecRel} is enough to make the two theories equivalent. Thus, it follows that the conceptual distance between relativistic and classical kinematics is $1$.

\begin{thm}\label{dsrck}
Classical and relativistic kinematics are distinguished from each other by only one concept, namely the existence of some distinguished observers captured by formula \eqref{E} below, i.e., $\Cd_{\omega}^{\omega}(\ax{ClassicalKin},\ax{SpecRel})=1$.\footnote{It is worth noting that in the proof of Theorem~\ref{dsrck}, we add only a unary concept $E$ to $\ax{SpecRel}$ to get a theory definitionally equivalent to $\ax{ClassicalKin}$.}
\end{thm}

\begin{proof}
The key to this result is the surprising theorem stating that the only concept which needs to be added to \ax{SpecRel} to make it definitional equivalent to \ax{ClassicalKin} is a concept distinguishing a set of observers that are ``being at absolute rest'' as shown in \cite[p.72]{diss} and \cite[p.110]{ClassRelKin}.  Let $E$ be a unary relation symbol corresponding to this basic concept. Axiom \ax{AxPrimitiveEther}, see \cite[p.46]{diss} and \cite[p.87]{ClassRelKin}, defines $E$ as follows:
\begin{equation}\label{E}
  \exists  v_0 \big[\mathsf{IOb}(v_0) \land \forall  v_1 \big(E(v_1)\leftrightarrow [\mathsf{IOb}(v_1)\land \varphi(v_0,v_1)]\big)\big],
\end{equation}
where $\mathsf{IOb}$ is a unary relation symbol represents inertial observers and $\varphi(v_0,v_1)$ is a formula in the language of \ax{SpecRel} capturing that observers $v_0$ and $v_1$ are stationary with respect to each other. In this proof, we only need that $\varphi(v_0,v_1)$ is a formula with two free variables in the language of $\ax{SpecRel}$; its concrete definition plays no rule here.  Let $$\ax{SpecRel^E}=\ax{SpecRel}\cup\{\ax{AxPrimitiveEther}\}.$$

First, we need to prove that $\ax{SpecRel}\leadsto \ax{SpecRel^E}$. To do so, it is enough to show that $\ax{SpecRel^E}$ is a conservative extension of $\ax{SpecRel}$, i.e., $\ax{SpecRel}\sqsubseteq\ax{SpecRel^E}$, because the languages of these theories differ only in the unary relation symbol $E$. So, we need to show that for any formula $\rho$ of the language of $\ax{SpecRel}$, 
\[
\ax{SpecRel}\models \rho \iff \ax{SpecRel^E}\models \rho.
\]  
Let $\rho$ be an arbitrary formula of the language of $\ax{SpecRel}$. Since $\ax{SpecRel}\subseteq \ax{SpecRel^E}$, $\ax{SpecRel}\models \rho$ implies $\ax{SpecRel^E}\models \rho$.  We prove the other direction by proving that, if $\ax{SpecRel} \not\models \rho$, then $\ax{SpecRel}^E \not\models \rho$. Let $\mathfrak{M}$ be a model of $\ax{SpecRel}$. Since $\ax{SpecRel} \models \exists v_0 \mathsf{IOb}(v_0)$,  there exists an $a\in \mathsf{IOb}^{\mathfrak{M}}$. Let us fix such element $a$
of $\mathsf{IOb}^{\mathfrak{M}}$ and let extension $\mathfrak{M}'$ of $\mathfrak{M}$ be defined by adding the following relation  to
$\mathfrak{M}$:
$$E^{\mathfrak{M}'}=\left\{b \in \mathsf{IOb}^{\mathfrak{M}}:\exists\text{ an assignment }\tau \big[\tau(v_0)=a,\tau(v_1)=b\text{ and }\mathfrak{M},\tau\models\varphi\big]\right\},$$
where $\varphi^{\mathfrak{M}}$ is the binary relation defined by formula $\varphi(v_0,v_1)$ in model $\mathfrak{M}$. By construction,  $\mathfrak{M}'$ is a model of $\ax{SpecRel^E}$.
Therefore, if $\mathfrak{M} \models \neg  \rho$, then $\mathfrak{M'}\models \neg \rho$, because $\mathfrak{M}'$ is an extension of $\mathfrak{M}$ means that $\Th(\mathfrak{M})\sqsubseteq\Th(\mathfrak{M}')$. Consequently, $\ax{SpecRel} \not\models \rho$ implies $\ax{SpecRel^E} \not\models \rho$, which is what we wanted to prove. This completes the proof of $\ax{SpecRel}\leadsto \ax{SpecRel^E}$, and hence $$\Cd_{\omega}^{\omega}(\ax{SpecRel},\ax{SpecRel^E})\le 1.$$ 
By Corollary 9 in \cite[p.72]{diss} and \cite[p.110]{ClassRelKin}, \ax{SpecRel^E} is definitionally equivalent to \ax{ClassicalKin}. Hence, 
$$\Cd_{\omega}^{\omega}(\ax{SpecRel^E}, \ax{ClassicalKin})=0.$$ 
Therefore, 
 $$\Cd_{\omega}^{\omega}(\ax{SpecRel},\ax{ClassicalKin})\le \Cd_{\omega}^{\omega}(\ax{SpecRel},\ax{SpecRel^E})+\Cd_{\omega}^{\omega}(\ax{SpecRel^E}, \ax{ClassicalKin})=1.$$
Moreover,  $\Cd_{\omega}^{\omega}(\ax{SpecRel},\ax{ClassicalKin})$ cannot be $0$ since $\ax{SpecRel}$ and $\ax{ClassicalKin}$ are not definitionally equivalent, see Theorem 5 in \cite{diss} or \cite{ClassRelKin}. 
Consequently, 
$$\Cd_{\omega}^{\omega}(\ax{SpecRel},\ax{ClassicalKin})=1$$ 
and thus the desired follows. 
\end{proof}
There are several ways how one can capture the structures of relativistic and classical kinematics in first order logic. Let us now
introduce another way to capture these theories. Let $\mathbb{R}$ be the set of all real numbers. Let $\mathsf{Ph}\subseteq\mathbb{R}^4\times\mathbb{R}^4$ be such that $(\bar{x},\bar{y})\in \mathsf{Ph}$ iff 
coordinate points $\bar x$ and $\bar y$ can be connected by a light signal, i.e., if $(x_1-y_1)^2-(x_2-y_2)^2-(x_3-y_3)^2-(x_4-y_4)^2=0$. Let $\mathsf{S}\subseteq\mathbb{R}^4\times\mathbb{R}^4$ be the simultaneity relation, i.e., $(\bar{x},\bar{y})\in S$ iff $x_1=y_1$. Consider the models $\mathfrak{R}=\langle\mathbb{R}^4,\mathsf{Ph}\rangle$ and $\mathfrak{N}=\langle\mathbb{R}^4,\mathsf{S}, \mathsf{Ph}\rangle$, these models capture the structure of special relativity and classical kinematics, respectively.

Let $T_n=\Th(\mathfrak{N})$ and $T_r=\Th(\mathfrak{R})$. Note that $T_n$ is in fact a conservative extension of $T_r$ and the conceptual distance between them is $1$, i.e., $\Cd_{\omega}^{\omega}(T_n,T_r)=1$. 

\begin{problem}[Hajnal Andr\'eka]Let $\X$ be the class of all theories $T$ such that $T_r$ is faithfully interpreted into $T$ and $T$ is faithfully interpreted into $T_n$. Is the following true: For all $T\in X$, $$\Cd_{\omega}^{\omega}(T_r,T)+\Cd_{\omega}^{\omega}(T,T_n)=1?$$  
\end{problem}

If the answer to the question in the above problem is yes, then no matter which classical (i.e., $T_n$-definable, but not $T_r$-definable) concept we add to special relativity ($T_r$) we will get classical kinematics ($T_n$). That would be an interesting insight for better understanding the connection between classical and relativistic concepts.

The investigation in this section opens so many questions: For any two concrete theories of physics, what is the conceptual distance between them? By Theorem~\ref{dsrck}, relativistic and classical kinematics are of conceptual distance one. However, the question ``what is the distance between relativistic and classical dynamics?'' remains open. Another natural related open problem is the following.
\begin{problem}[Jean Paul Van Bendegem]
What is the conceptual distance between classical and statistical thermodynamics?
\end{problem}

Of course, any answer to the above problems depends on the chosen axiomatizable theories capturing the physical theories in question. For an axiomatic approach of these thermodynamics theories, one can see, e.g., \cite{Caratheodory}, \cite{Cooper} and \cite{LiebXngvason}. 

\section{Ideas for other distances}\label{Other distances}
Interpreting a theory into another one is a fundamental concept in logic. In the following definition, we define some distance that uses faithful interpretations as the minimal step between theories. 

\begin{definition}For any two theories $T_1$ and $T_2$, we write $T_1 \ \mathbb{I} \ T_2$ iff one of these theories can be interpreted faithfully into the other one. Let $\X$ be an arbitrary class of theories in $\Lg_{\alpha}^{\beta}$. The  \emph{\textbf{faithful interpretation distance} on $\X$} is defined as the step distance on the cluster network $(\X,\equiv,\mathbb{I})$.
\end{definition}

Another natural idea one may have for defining a distance between theories is using a step that collapses two concepts into one, i.e., using the symmetric closure of following  relation for the minimal single steps. 

\begin{definition}Let $T,T'$ be two theories. We say that \emph{$T'$ is the resultant of $T$ after \textbf{collapsing two concepts}} iff $T'\equiv T\cup\{\varphi\leftrightarrow\psi\}$, for some $\varphi,\psi\in \Fm$.      
\end{definition}

\begin{example}The theory of abelian groups (in ordinary first order logic) is the resultant of the theory of all groups after collapsing the concepts $a\cdot b=c$ and $b\cdot a=c$.  
\end{example}

One can easily see that collapsing two concepts of a theory $T$ is a special case of adding an axiom to $T$. The converse is also true; adding an axiom $\varphi$ to $T$ is equivalent to collapsing the concept $\varphi$ with any theorem of $T$. So using the symmetric closure of the above relation for generating a distance will give the axiomatic distance (Definition~\ref{def-axiomatic-distance}).

\subsection{Dropping symmetry} In several cases, it might be natural not to assume the symmetry of distances between theories. For example, any inconsistent theory is understood to be of axiomatic distance $1$ from any consistent theory; we just need to add a contradiction as an axiom. But starting from an inconsistent theory, we can never reach a consistent theory by adding axioms; so considering this distance to be $\infty$ seems more natural.

Now, let us mimic the work of section~\ref{sd} under the consideration that symmetry is not required. For instance, a \emph{\textbf{directed cluster network}} is a triple $(\X,\E,\R)$, where $(\X,\E)$ is a cluster and $\R$ is an arbitrary relation on $\X$. For directed cluster network $\mathcal{X}=(\X,\E,\R)$, the \emph{\textbf{directed step distance}} $\ddd_{\mathcal{X}}:\X\times\X\rightarrow\NN\cup\{\infty\}$ can be defined completely analogously to step distance. With a very similar argument to the proof of Theorem~\ref{thm:min-path}, we can see that the following are true: For each $T_1,T_2,T_3\in \X$,  
\begin{enumerate}
\item[(a)] $\ddd_{\mathcal{X}}(T_1,T_2)\geq 0$ and $\ddd_{\mathcal{X}}(T_1,T_2)=0\iff T_1 \, \E \, T_2$.
\item[(b)] $\ddd_{\mathcal{X}}(T_1,T_2)\leq \ddd_{\mathcal{X}}(T_1,T_3)+ \ddd_{\mathcal{X}}(T_3,T_2)$.
\end{enumerate}

\begin{definition}Recall the axiom adding relation $\leftarrow$ introduced on page~\pageref{implication-theories} herein. Let $\X$ be a class of theories in logic $\Lg_{\alpha}^{\beta}$, then $(\X,\equiv,\leftarrow)$ is a directed cluster network. Its directed step distance is called the \emph{\textbf{directed axiomatic distance} on $\X$}.
\end{definition}

Clearly, Theorem~\ref{thm:axdist} is no longer true if we replace the axiomatic distance by the mimicked directed one. So, this is one of the situations where dropping the symmetry might be more interesting. Not just the axiomatic distance can be directed, but also the conceptual distance, faithful interpretation distance, and so on. It might be more appropriate to call these distances ``uni-directed distances'', indeed in the directed cluster network the single steps are determined by only one relation. We can also define \emph{bi-directed distances} or \emph{multi-directed distances} where the single steps can be determined by two or more relations. 

For example, it might be useful to introduce a ``new conceptual distance'' that measures the minimum number of concepts needed to be added to or removed from one theory to reach the other theory up to definitional equivalence. We already have a notion for concept adding (recall $\leadsto$ in Definition~\ref{def:conpext}). A precise definition for concept removal is also required. So, in this case, two relations will indicate the single steps and such new conceptual distance must be a bi-directed step distance. The notion of concept-removal below is inspired by the idea how the concept of faster-than-light observers were removed from the theory capturing classical kinematics in \cite{diss} and \cite{ClassRelKin}.

\begin{definition}\label{one-concept-removal}
Let $T,T^{-}$ be two theories. We say that $T^-$ is a \emph{\textbf{concept-removal} of $T$} and we write $T\gtrdot T^-$ iff there is $\varphi\in \Fm$ such that $T^-=T_m\cup\{\neg \varphi\}$ for some maximal consistent subtheorey of $\Cn(T)$ for which $T_m\not\models \varphi$. 
\end{definition}

In some cases, it is possible to remove finitely many concepts in only one step. For instance, assume that $\alpha=0$ and $\beta=1$, and let $\LL_1=\{P_1,P_2\}$ and $\LL_2=\{P_1,P_2,P_3,P_4\}$ for some sentential constants $P_1,P_2,P_3$ and $P_4$. For each $i\in\{1,2\}$, let $T_i$ be the empty theory on the language $\LL_i$. Recall the proof of Theorem~\ref{thm:non-discrete}, it was shown that we need exactly two steps of adding concepts to get $T_2$ from $T_1$. However, removing a single concept, namely $(P_2\not=P_3)\lor(P_3\not=P_4)$, from $T_2$ gives $T_1$.

\begin{definition}\label{directed-conceptual-distance}
Let $\X$ be a class of theories and let $\R$ be the union of the relations $\leadsto$ and $\gtrdot$. The \emph{\textbf{bi-directed conceptual distance} on $\X$} is the directed step distance on $(\X,\intertrans,\R)$. In the case when $\X$ is the class of all theories of logic $\Lg_{\alpha}^{\beta}$, we denote this distance by $\dCd_{\alpha}^{\beta}$.
\end{definition}

The facts, in the paragraph before the above definition, that $\dCd_{\alpha}^{\beta}(T_1,T_2)=2$ and $\dCd_{\alpha}^{\beta}(T_2,T_1)=1$ show that the bi-directed distance is not symmetric. Hence, conceptual and bi-directed conceptual distances are different, but it is still interesting to understand how much different they are. For example, the following problem is worth investigating.

\begin{problem}
Are there two theories $T$ and $T'$ such that 
\begin{enumerate} 
\item[(1)] $\Cd_{\alpha}^{\beta}(T,T')$, $\dCd_{\alpha}^{\beta}(T,T')$ and $\dCd_{\alpha}^{\beta}(T',T)$ are all finite,
\item[(2)] $\Cd_{\alpha}^{\beta}(T,T')\not=\dCd_{\alpha}^{\beta}(T,T')$ and  $\Cd_{\alpha}^{\beta}(T,T')\not=\dCd_{\alpha}^{\beta}(T',T)$?
\end{enumerate}
\end{problem}

It is not straightforward to find this example, each conceptual distance calculated in the present paper coincides with one of its corresponding bi-directed conceptual distances. In order to do the same with axiomatic distance, i.e., to define a \emph{bi-directed axiomatic distance}, all we need is a precise definition for theorem removal. We propose the following theorem-removal notion.

\begin{definition}
Let $T,T^{-}$ be two theories. We say that $T^-$ is a \emph{\textbf{theorem-removal} of $T$} iff there is $\varphi\in \Fm$ such that $T\models\varphi$ and $T^-$ is maximal consistent subtheorey of $\Cn(T)$ for which $T^-\not\models \varphi$.
\end{definition}

\section{Concluding philosophical remarks} One very important topic in the philosophy of science, is how different scientific theories can be compared to each other, especially in the case of competing theories. The first criterion for theory comparison is empirical adequacy. One theory is better than another if it accounts for more of the data or phenomena than another. In the past, this has often been, or has been presented as being, a fairly straightforward matter to decide by philosophers and historians.

This is oversimplified for two reasons. One is that sometimes in the history of science, one theory accounts for some of the data or phenomena very well. Another accounts for another area of data very well. They both agree on, but have different accounts of, the same data and both have failings. The comparison of two theories in terms of empirical adequacy requires that we count the data or the phenomena. Deciding what to count, and how to assign weight to it, has some arbitrariness to it. This is very well illustrated in \cite{chang} where he discusses the history of the competition between the phlogiston theory of water and the compound theory of water. With Chang, we conclude that deciding that one theory is more empirically adequate than another is not at all times, and in all circumstances, simple and straightforward, and with the fragmentation of science into more and more specialized areas of research, it is increasingly rare to find empirical adequacy to be enough to decide between competing theories. 

Worse: with more complicated and ``cutting-edge'' examples, we find that observation in science is not simple, but is informed by instruments, theory and language; making an observation is an informed and educated act. At the edges of science, we make observations using highly specialized instruments, which are constructed based on their own theories. Thus, what looked in the past to be a relatively simple judgment to make: ``this scientific theory is better than this other'' turns out to be rather subtle; since it requires individuation and assigning weights, in a way that is independent with respect to the theories themselves. 

Under an over simplified view of the unity of science, the subtlety threatens any pretense science has to objectivity, because what counts as a true and verifiable statement takes specialized instruments that we assume to work according to the theory we have of the instrument. If someone has an alternative account, then the explanation for the phenomenon changes. As a result, if we want to recover some semblance of objectivity in science, it is ever more pressing to receive confirmation of a theory from other directions independent of the theory.

In terms of objectivity, one reassuring feature of science is its precision. Logic is the most precise form of investigation. Under the pressure of our considerations above, when we have several logical theories that are each to some extent empirically adequate, it is not clear that we should retain one and dismiss the other. We then have pluralism in science. Pluralism in science is an obvious philosophical position when we consider that several theories are all more-or-less empirically adequate, and show merits with respect to other, incomparable, or only artificially comparable, virtues and vices. The virtues might include: simplicity (determined by language, proofs, metaphysical parsimony or concepts), meeting a particular goal of the scientists, neatness of categorization, breadth of categorisation, standardization of explanation or meeting operational opportunities and so on. Vices might include: complexity (determined by language, proofs, metaphysical elaborations that have little use within a theory or concepts), goal failure, messy or narrow categorization, non-standard explanations, being too ambitious and not going outside the constraints of the operations available. With this plethora of incommensurable theories all competing and each adequate in their own way and for their own purpose, we look elsewhere than between the theories in and of themselves and the data to make sense of the present state of science. What we then look for are other ways of comparing theories, while accepting them until such time as we come up against a good reason to give one up, such as: its being refuted by new evidence or being too remote from too many other theories to be worth pursuing (now).

For this reason, the relations between theories, independent of their relation to reality, becomes very important. Until now, this area of study has mostly been \textit{qualitative}. In the present paper, we explore a new \textit{quantitative} approach: the measured distance between theories. To establish this distance, we need to study the structure of the differences, i.e., the connections, between theories. By developing several such metrics based on our definitions, and noticing that some are less interesting than others, we already learn a lot. Counting axioms does not give us much information about the distance between theories. Counting concepts is much more subtle and informative.

One area of study that has a close relationship with the notion of conceptual distance is that of complexity. As we know, complexity, also can be measured in several ways: Turing complexity, in terms of the analytic hierarchy, and so on. If one theory is more complex than another in one of these measures, then it is natural to investigate the relationship between that and the distances we look at here. Some of the significance of the present work might be in its relationship to complexity theory. This is a subject of future investigation.

The idea of having a notion of distance between theories (of the same nature) seems applicable in any science. In computer science, programming languages and other systems can be seen as axiomatized theories. For more details about this, see, e.g., \cite{floyd}, \cite{hoare} and \cite{meyer}. Hence, it seems also natural to search for the best fit notion of equivalence between these theories. Developing this may give us insight to determine what can be one step difference between two such theories. Having these in mind, a distance can be then defined in the same way of section~\ref{sd} herein. The novelty here would be in choosing such equivalence and one step relation in a way that guarantees that the corresponding step distance is applicable.

\section*{Acknowledgement.} We shall thank Hajnal Andr\'eka, Jean Paul Van Bendegem, Zal\'an Gyenis, P\'eter Juh\'asz, Istv\'an N\'emeti and others for their fruitful comments and  questions. We are especially grateful for Judit X. Madar\'asz because of her insightful comments leading us to significantly improve the paper.

\bibliographystyle{apalike}
\bibliography{LogRel12017}

\address{
\begin{minipage}{0.5\textwidth}
MICH{\`E}LE FRIEND\\
Columbian College of Arts \& Sciences\\  
George Washington University\\
{\it E-mail}: michele@gwu.edu
\end{minipage}
\begin{minipage}{0.5\textwidth}
MOHAMED KHALED\\
Alfr{\' e}d R{\' e}nyi Institute of Mathematics\\
Hungarian Academy of Sciences\\
{\it E-mail}: khaled.mohamed@renyi.mta.hu
\end{minipage}
\vspace{0.5cm}\\
\begin{minipage}{0.5\textwidth}
KOEN LEFEVER\\
Centre for Logic and Philosophy of Science\\
Vrije Universiteit Brussel\\
{\it E-mail}: koen.lefever@vub.be
\end{minipage}
\begin{minipage}{0.5\textwidth}
GERGELY SZ{\' E}KELY\\
Alfr{\' e}d R{\' e}nyi Institute of Mathematics\\
Hungarian Academy of Sciences\\
{\it E-mail}: szekely.gergely@renyi.mta.hu
\end{minipage}
}
\end{document}